\documentclass[12pt,reqno]{amsart}
\usepackage{amsmath, amsthm, amssymb}

\advance \topmargin by -\headheight
\advance \topmargin by -\headsep
     
\evensidemargin \oddsidemargin
\newtheorem{theorem}{Theorem}[section]
\newtheorem{lemma}[theorem]{Lemma}

\newtheorem{proposition}[theorem]{Proposition}

\theoremstyle{definition}

\theoremstyle{remark}

\numberwithin{equation}{section}

\newcommand{\mmod}[1]{\,\,(\text{\rm{mod}}\,\,#1)}
\newcommand{\nmod}[1]{\,\, \text{\rm{mod}}\,\,#1}

\def\bfm{{\mathbf m}}
\def\bfn{{\mathbf n}}

\def\bfy{{\mathbf y}}

\def\calI{{\mathcal I}}
\def\calJ{{\mathcal J}}

\def\calM{{\mathcal M}}
\def\calN{{\mathcal N}}

\def\calQ{{\mathcal Q}}

\def\calT{{\mathcal T}}

\def\calX{{\mathcal X}} \def\calY{{\mathcal Y}}
\def\calZ{{\mathcal Z}}

\def\btil{\widetilde{b}}
\def\bfr{{\mathbf r}}

\def\dbN{{\mathbb N}}
\def\dbR{{\mathbb R}} \def\What{{\widehat W}}
\def\dbZ{{\mathbb Z}}

\def\grB{{\mathfrak B}}
\def\grC{{\mathfrak C}}

\def\grJ{{\mathfrak J}}

\def\grm{{\mathfrak m}}\def\grM{{\mathfrak M}}
\def\grS{{\mathfrak S}}
\def\grW{{\mathfrak W}}\def\grB{{\mathfrak B}}\def\grC{{\mathfrak C}}

\def\grT{{\mathfrak T}}

\def\grw{{\mathfrak w}}\def\grW{{\mathfrak W}}

\def\alp{{\alpha}}  
\def\bet{{\beta}}  
\def\gam{{\gamma}}   

\def\del{{\delta}} \def\Del{{\Delta}}

\def\tet{{\theta}}  \def\Tet{{\Theta}}
\def\kap{{\kappa}}
\def\lam{{\lambda}} \def\Lam{{\Lambda}} 
 \def\bfchi{{\boldsymbol \chi}}

\def\sig{{\sigma}}  

\def\Ups{{\Upsilon}} 
\def\phi{{\varphi}}

\def\chibar{{\overline \chi}}
\def\ome{{\omega}}  
\def\d{{\partial}}
\def\eps{\varepsilon}

\def\le{\leqslant} \def\ge{\geqslant}

\def\d{{\,{\rm d}}}

\begin{document}
\title[On sums of powers of almost equal primes]{On sums of powers of almost equal primes}
\author[Bin Wei]{Bin Wei$^*$}
\address{School of Mathematics, Shandong University, Jinan, Shandong 250100, China}
\email{bwei.sdu@gmail.com}
\author[Trevor D. Wooley]{Trevor D. Wooley}
\address{School of Mathematics, University of Bristol, University Walk, Clifton, Bristol BS8 1TW, United 
Kingdom}
\email{matdw@bristol.ac.uk}
\thanks{Bin Wei is grateful to the China Scholarship Council (CSC) for supporting his studies in the United 
Kingdom}
\date{}
\begin{abstract} We investigate the Waring-Goldbach problem of representing a positive integer $n$ as the 
sum of $s$ $k$th powers of almost equal prime numbers. Define $s_k=2k(k-1)$ when $k\ge 3$, and put 
$s_2=6$. In addition, put $\tet_2=\frac{19}{24}$, $\tet_3=\frac{4}{5}$ and $\tet_k=\frac{5}{6}$ 
$(k\ge 4)$. Suppose that $n$ satisfies the necessary congruence conditions, and put $X=(n/s)^{1/k}$. We 
show that whenever $s>s_k$ and $\eps>0$, and $n$ is sufficiently large, then $n$ is represented as the 
sum of $s$ $k$th powers of prime numbers $p$ with $|p-X|\le X^{\tet_k+\eps}$. This conclusion is 
based on a new estimate of Weyl-type specific to exponential sums having variables constrained to short 
intervals.
\end{abstract}
\maketitle

\section{Introduction} A formal application of the circle method suggests that whenever $s$ and $k$ are 
natural numbers with $s\ge k+1$, then all large integers $n$ satisfying appropriate local conditions should 
be represented as the sum of $s$ $k$th powers of prime numbers. With this expectation in mind, consider 
a natural number $k$ and prime $p$, take $\tau=\tau(k,p)$ to be the integer with $p^\tau|k$ but 
$p^{\tau+1}\nmid k$, and then define $\gam=\gam(k,p)$  by putting $\gam(k,p)=\tau+2$, when 
$p=2$ and $\tau>0$, and otherwise $\gam(k,p)=\tau+1$. We then define $R=R(k)$ by putting 
$R(k)=\prod p^\gam $, where the product is taken over primes $p$ with $(p-1)|k$. In 1938, Hua 
\cite{Hua1938, Hua1965} established that whenever $s>2^k$, and $n$ is a sufficiently large natural 
number with $n\equiv s\mmod{R}$, then the equation
\begin{equation}\label{1.1}
p_1^k+p_2^k+\ldots +p_s^k=n
\end{equation}
is soluble in prime numbers $p_j$. The congruence condition here excludes degenerate situations in which 
variables might otherwise be forced to be prime divisors of $k$. An intensively studied refinement of Hua's 
theorem is that in which the variables are constrained to be almost equal. Writing $X=(n/s)^{1/k}$, one 
seeks an analogue of Hua's theorem in which the variables $p_j$ satisfy $|p_j-X|\le Y$, with $Y$ rather 
smaller than $X$. Although limitations in our knowledge concerning the distribution of primes constrain 
such investigations to intervals with $Y\ge X^{7/12+\eps}$ or thereabouts (see \cite{Hux1972}), for larger 
values of $k$, previous authors have obtained such conclusions only for $Y\ge X^{1-\Del}$, with 
$\Del$ extremely small. In this paper we decisively improve such conclusions, showing that for large 
$s$, one may take $\Del$ large for all $k$.\par

In order to facilitate further discussion, we introduce some additional notation. We say that the 
exponent $\Del_{k,s}$ is {\it admissible} when, provided that $\Del$ is a positive number with 
$\Del<\Del_{k,s}$, then for all sufficiently large positive integers $n$ with $n\equiv s\mmod{R}$, 
the equation (\ref{1.1}) has a solution in prime numbers $p_j$ satisfying 
$|p_j-X|\le X^{1-\Del}$ $(1\le j\le s)$. Old work of Wright \cite{Wri1937} on Waring's problem shows 
that admissible exponents $\Del_{k,s}$ must always satisfy the condition 
$0\le \Del_{k,s}\le \tfrac{1}{2}$.\par

Attention has naturally focused in the first instance on the situation for smaller values of $k$. We 
note in this context that, as a consequence of Hua's theorem, all large integers congruent to $5$ 
modulo $24$ are the sum of five squares of prime numbers, and all large odd integers are the sum of 
nine cubes of prime numbers. The first breakthrough was made by Liu and Zhan \cite{LZ1996}, who 
in the former setting showed, subject to the truth of the Generalized Riemann Hypothesis (GRH), that 
the exponent $\Del_{2,5}=\frac{1}{10}$ is admissible. Subsequently, they introduced an approach to 
treating enlarged major arcs \cite{LZ1998}, and this allowed Liu, L\"u and Zhan 
\cite[Theorem 1.3]{LLZ2006} to establish the same conclusion unconditionally. The sharpest unconditional 
result at present is due to Kumchev and Li \cite[Theorem 1]{KL2012}, who prove that 
$\Del_{2,5}=\frac{1}{9}$ is admissible. Moreover, when more summands are available, the latter authors 
show \cite[Theorem 5]{KL2012} that one has the admissible exponent
$$\Del_{2,s}=\begin{cases} \frac{9}{40}\left(\frac{s-4}{s-3}\right),&\text{when $6\le s\le 16$,}\\
\frac{5}{24},&\text{when $s\ge 17$.}\end{cases}$$
We refer the reader to \cite{Bau1997, Bau2005, BW2006, LW2008, LZ2000, Lv2005, Lv2006, Men2006} 
for further results interpolating those already cited.\par

Turning next to sums of cubes and higher powers, Meng \cite{Men1997} showed that 
$\Del_{3,9}=\frac{1}{66}$ is admissible, subject to the truth of GRH, and L\"u and Xu 
\cite[Theorem 1]{LX2007} established this conclusion unconditionally. In general, again subject to the 
truth of GRH, Meng \cite{Men2002} has shown that the exponent
$$\Del_{k,s}=\frac{1}{(k-1)2^{2k-1}+2}$$
is admissible whenever $2\le k\le 10$ and $s>2^k$, and Sun and Tang \cite[Theorem 2]{ST2009} have 
established this conclusion unconditionally. It is apparent that, in general, these admissible exponents 
remain small even when $s$ is large.\par

When $k\ge 2$, we define the integer $t_k$ by putting
\begin{equation}\label{1.2}
t_k=\begin{cases} 3,&\text{when $k=2$,}\\
k(k-1),&\text{when $k\ge 3$,}\end{cases}
\end{equation}
and define the real number $\tet_k$ by putting
\begin{equation}\label{1.w1}
\tet_k=\begin{cases} \tfrac{19}{24},&\text{when $k=2$,}\\
\tfrac{4}{5},&\text{when $k=3$,}\\
\tfrac{5}{6},&\text{when $k\ge 4$.}\end{cases}
\end{equation}
The main result of this paper shows that there are large admissible exponents $\Del_{k,s}$ as soon 
as $s>2t_k$.

\begin{theorem}\label{theorem1.1}
Let $s$ and $k$ be integers with $k\ge 2$ and $s>2t_k$. Suppose that $\eps>0$, that $n$ is a 
sufficiently large natural number satisfying $n\equiv s\mmod{R}$, and write $X=(n/s)^{1/k}$. Then the 
equation $n=p_1^k+p_2^k+\ldots+p_s^k$ has a solution in prime numbers $p_j$ with 
$|p_j-X|\le X^{\tet_k+\eps}$ $(1\le j\le s)$.
\end{theorem}

This theorem shows that the exponent $\Del_{k,s}=\tfrac{1}{6}$ is admissible whenever $k\ge 2$ and 
$s>2t_k$. In contrast to the admissible exponents derived in the previous work cited above, this 
exponent is bounded away from zero as $k\rightarrow \infty$. Moreover, only when $k=2$ and $s\ge 17$ 
does previous work (of Kumchev and Li \cite{KL2012}) match our new conclusions. We remark that since 
$\tfrac{19}{24}=\tfrac{1}{2}\left( 1+\tfrac{7}{12}\right)$, this exponent is in some sense half way 
between the trivial exponent $1$, and the exponent $\tfrac{7}{12}$ that, following the work of Huxley 
\cite{Hux1972}, represents the effective limit of our knowledge concerning the asymptotic distribution of 
prime numbers in short intervals.\par

Aficionados of the circle method will anticipate that similar conclusions may be established in 
problems with fewer variables if one seeks instead conclusions valid only for {\it almost all} 
integers $n$, so that there are at most $o(N)$ exceptional integers $n$ not exceeding $N$, as 
$N\rightarrow \infty$. We say that the exponent $\Del^*_{k,s}$ is {\it semi-admissible} when, 
provided that $\Del$ is a positive number with $\Del<\Del^*_{k,s}$, then for almost all positive 
integers $n$ with $n\equiv s\mmod{R}$, the equation (\ref{1.1}) has a solution in prime numbers 
$p_j$ satisfying $|p_j-X|\le X^{1-\Del}$ $(1\le j\le s)$. 

\begin{theorem}\label{theorem1.2}
Let $s$ and $k$ be integers with $k\ge 2$ and $s>t_k$, and suppose that $\eps>0$. Then for almost all 
positive integers $n$ with $n\equiv s\mmod{R}$ (and, in case $k=3$ and $s=7$, satisfying also 
$9\nmid n$), the equation $n=p_1^k+p_2^k+\ldots+p_s^k$ has a solution in prime numbers $p_j$ 
with $|p_j-X|\le X^{\tet_k+\eps}$ $(1\le j\le s)$, where $X=(n/s)^{1/k}$.
\end{theorem}

We note that the additional condition $9\nmid n$ in the case $k=3$ and $s=7$ is required to ensure 
the solubility of (\ref{1.1}) modulo $9$. Previous work on this topic has focused on smaller $k$. So 
far as sums of four squares of primes are concerned, L\"u and Zhai \cite{LZ2009} showed that the 
exponent $\Del_{2,4}^*=\tfrac{4}{25}$ is semi-admissible. Kumchev and Li \cite[Theorem 3]{KL2012} 
improved this conclusion, showing that $\Del_{2,4}^*=\frac{9}{50}$ is semi-admissible. Theorem 
\ref{theorem1.2} improves this result further, showing in particular that $\Del_{2,4}^*=\frac{5}{24}$ 
is semi-admissible\footnote{A result of Li and Wu \cite[Theorem 3]{LW2008}, tantamount to the assertion 
that $\Del_{2,4}^*=\frac{9}{40}$ is semi-admissible, contains an infelicity discussed in \S8 below.}. 
Considering next sums of seven or eight cubes of primes (with the additional local solubility 
condition for seven cubes of primes implied), Liu and Sun \cite[Theorem 1]{LS2012} showed that the 
exponents $\Del_{3,7}^*=\tfrac{3}{38}$ and $\Del^*_{3,8}=\tfrac{1}{10}$ are 
semi-admissible\footnote{The argument of Zhao \cite[Theorem 1.2]{Zha2010}, underlying the assertion 
that $\Del_{3,7}^*=\frac{3}{25}$ and $\Del_{3,8}^*=\frac{2}{15}$ are semi-admissible, likewise 
contains an infelicity discussed in \S8 below.}. We note also that the recent work of Tang and Zhao 
\cite[Theorem 1]{TZ2013} shows in particular that the exponent $\Del_{4,13}^*=\tfrac{5}{202}$ is 
semi-admissible. Theorem \ref{theorem1.2}, on the other hand, obtains the considerably stronger 
semi-admissible exponents $\Del_{3,s}^*=\tfrac{1}{5}$ for $s\ge 7$ and 
$\Del_{4,s}^*=\tfrac{1}{6}$ for $s\ge 13$. Indeed, it follows from our new theorem that whenever 
$k\ge 4$ and $s>k(k-1)$, then the exponent $\Del_{k,s}^*=\tfrac{1}{6}$ is always semi-admissible.
\par

We outline our proof of Theorem \ref{theorem1.1}, which proceeds via the circle method, in \S2. By 
comparison with previous treatments, this argument contains two novel features. The first is an 
estimate for moments of exponential sums over $k$th powers in short intervals, of order $2s$, that 
achieves essentially optimal estimates as soon as $s\ge t_k$. This serves as a substitute for the 
traditional use of Hua's lemma, though for problems involving short intervals is considerably sharper. 
In \S3 we explain how this estimate follows from the analogous work of Daemen 
\cite{Dae2010a, Dae2010b}, based on his use of the so-called {\it binomial descent method}. The 
second novel feature is a substitute for a Weyl-type estimate for exponential sums over variables in 
short intervals that delivers non-trivial estimates on the minor arcs in a Hardy-Littlewood dissection 
even when the corresponding major arcs are rather narrow. This estimate again makes use of 
Daemen's estimates via a bilinear form treatment motivated by analogous arguments making use of 
Vinogradov's mean value theorem. This argument is described in \S4. Both the work in \S3 and that 
in \S4 makes heavy use of the latest work \cite{Woo2014} concerning Vinogradov's mean value theorem. 
Our sharpest conclusions for $k=2$ and $3$ require a discussion of estimates of Weyl type particular to 
these exponents, and this we record in \S5. The analysis of the major arc estimates required in our 
application of the circle method is discussed, in stages, in \S6, 7 and 8. Finally, we discuss exceptional 
set estimates in \S9, thereby establishing Theorem \ref{theorem1.2}.\par

Throughout this paper, the letter $\eps$ will refer to a small positive number. We adopt the convention 
that whenever $\eps$ occurs in a statement, then the statement holds for all sufficiently small $\eps>0$. 
Similarly, we write $L$ for $\log X$, and adopt the convention that whenever $L^c$ occurs in a statement, 
then the statement holds for some $c>0$. In addition, as usual, we write $e(z)$ for $e^{2\pi iz}$. Finally, 
we summarize an inequality of the shape $M<m\le 2M$ by writing $m\sim M$.\par

The authors are very grateful to Professors Angel Kumchev and Roger Baker for pointing out a significant 
oversight in \S4 of the original version of this work that has been corrected in this manuscript.

\section{Outline of the method} Our basic approach to the application of the circle method is 
straightforward so far as the Waring-Goldbach problem with almost equal summands is concerned. 
Suppose that $k$ and $s$ are integers with $k\ge 2$ and $s\ge t_k$, where $t_k$ is defined as in 
(\ref{1.2}). Write 
\begin{equation}\label{2.0}
K=\begin{cases} 36,&\text{when $k=2$,}\\
2t_k(t_k+2),&\text{when $k>2$.}
\end{cases}
\end{equation}
Let $\tet$ be a real number with $\tet_k<\tet<1$, and let $\del$ be a sufficiently small, but fixed, 
positive number with $4K\del<\min\{\tet-\tet_k,1-\tet\}$. Consider a sufficiently large natural 
number $N$, put $X=(N/s)^{1/k}$, and write $Y=X^\tet$. When $n$ is a natural number with 
$N\le n\le N+X^{k-1}Y$, we denote by $\rho_s(n)$ the weighted number of solutions of the equation 
(\ref{1.1}) with $|p_i-X|\le Y$ $(1\le i\le s)$ given by
$$\rho_s(n)=\underset{p_1^k+\ldots +p_s^k=n}{\sum_{|p_1-X|\le Y}\ldots \sum_{|p_s-X|\le Y}}
(\log p_1)\ldots (\log p_s).$$
Define
\begin{equation}\label{2.1}
f(\alp)=\sum_{|p-X|\le Y}(\log p)e(p^k\alp),
\end{equation}
where the summation is over prime numbers $p$. Then it follows from orthogonality that
\begin{equation}\label{2.2}
\rho_s(n)=\int_0^1f(\alp)^se(-n\alp)\d\alp .
\end{equation}

\par Next we define the Hardy-Littlewood dissection underpinning our application of the circle method. 
Write
\begin{equation}\label{2.3}
P=X^{2K\del}\quad \text{and}\quad Q=X^{k-2}Y^2P^{-1}.
\end{equation}
We denote by $\grM$ the union of the major arcs
$$\grM(q,a)=\{ \alp\in [0,1):|q\alp-a|\le Q^{-1}\},$$
with $0\le a\le q\le P$ and $(a,q)=1$. Finally, we write $\grm=[0,1)\setminus \grM$ for the set of minor 
arcs complementary to the set of major arcs $\grM$. When $\grB$ is a measurable subset of $[0,1)$, we 
now define
\begin{equation}\label{2.4}
\rho_s(n;\grB)=\int_\grB f(\alp)^se(-n\alp)\d\alp .
\end{equation}
Thus, since $[0,1)$ is the disjoint union of $\grM$ and $\grm$, one finds from (\ref{2.2}) that
\begin{equation}\label{2.5}
\rho_s(n)=\rho_s(n;\grM)+\rho_s(n;\grm).
\end{equation}

\par The analysis of the major arc contribution $\rho_s(n;\grM)$ is essentially routine, though as is 
typical with the Waring-Goldbach problem with almost equal summands, the wide major arcs cause 
some technical difficulties. When $a\in \dbZ$, $q\in \dbN$ and $\bet\in \dbR$, define
$$v(\bet)=k^{-1}\sum_{(X-Y)^k\le m\le (X+Y)^k}m^{-1+1/k}e(\bet m)$$
and
$$S(q,a)=\sum^q_{\substack{r=1\\ (r,q)=1}}e(ar^k/q).$$
We then define the singular integral
$$\grJ(n)=\int_0^1v(\bet)^se(-\bet n)\d\bet $$
and the singular series
$$\grS(n)=\sum_{q=1}^\infty \phi(q)^{-s}\sum^q_{\substack{a=1\\ (a,q)=1}}S(q,a)^se(-na/q).$$

\par We temporarily proceed in greater generality than is demanded by our present choice of 
parameters, so as to permit future reference to our present discussion. The standard theory of major 
arc contributions in the Waring-Goldbach problem requires little modification to deliver satisfactory 
estimates for $\grJ(n)$ and $\grS(n)$ (see \cite[Lemma 8.12]{Hua1965} and 
\cite[Lemmata 6.4 and 8.3]{LZ2012}). Thus, whenever $s\ge 4$, $Y\ge X^{1/2+\del}$ and 
$|n-sX^k|\le X^{k-1}Y$, there exists a positive number $\grC$ for which 
$\grJ(n)=\grC Y^{s-1}X^{1-k}$. In addition, whenever $s>\max\{3,k(k-1)\}$ and 
$n\equiv s\mmod{R(k)}$ (and, in the case $k=3$ and $s=7$, one has in addition $9\nmid n$), there is a 
positive number $\eta=\eta(s,k)$ for which
$$1\ll \grS(n)\ll (\log X)^\eta.$$
Indeed, when $k=2$ and $s=4$, one may take $\eta=1$, and when $s\ge 5$ one is at liberty to take 
$\eta=0$. Thus we conclude that, in the circumstances at hand, one has
\begin{equation}\label{2.6}
Y^{s-1}X^{1-k}\ll \grS(n)\grJ(n)\ll Y^{s-1}X^{1-k}(\log X)^\eta.
\end{equation}

\par We summarise the analysis of the major arc contribution $\rho_s(n;\grM)$ in the form of a 
proposition.

\begin{proposition}\label{prop2.1} Suppose that $k\ge 2$ and $s\ge \min\{5,k+2\}$. Then, whenever 
$\tfrac{19}{24}<\tet<1$, $Y=X^\tet$ and $n$ is a natural number with $N\le n\le N+X^{k-1}Y$, one 
has
$$\rho_s(n;\grM)=\grS(n)\grJ(n)+O(Y^{s-1}X^{1-k}(\log X)^{-1}).$$
\end{proposition}

\begin{proof} The desired conclusion may be established by following the argument of the proof of 
\cite[Proposition 5.1]{KL2012}. The latter argument is superficially restricted to the case $k=2$ and 
$s=4$, but the generalisation to arbitrary exponents $k$ and $s\ge 5$ causes no extra difficulties. We 
provide details of the necessary argument in \S\S6--8 below. We note that compared to 
\cite[Proposition 5.1]{KL2012}, we have an additional weight $\log p$ within our definition (\ref{2.1}) 
of the exponential sum $f(\alp)$. This again is easily accommodated within the argument of the proof 
of \cite[Proposition 5.1]{KL2012}. The reader may find it illuminating to compare with the argument 
of the proof of \cite[Proposition 2.1]{ST2009}, which has the potential to establish the conclusion of our 
proposition subject to the more restricted hypotheses $s\ge 5$ and $Y\ge X^{4/5+\eps}$.
\end{proof}

Suppose that $k\ge 2$, $s>t_k$ and $Y>X^\tet$. Then by combining (\ref{2.6}) with the conclusion of 
Proposition \ref{prop2.1}, it follows that whenever $n\equiv s\mmod{R(k)}$ (and, in case $k=3$ and 
$s=7$, one has also $9\nmid n$), then
\begin{equation}\label{2.7}
\rho_s(n;\grM)\gg Y^{s-1}X^{1-k}.
\end{equation}

\par In order to estimate the minor arc contribution $\rho_s(n;\grm)$, in \S3 we prepare an analogue of 
Hua's lemma. We recall the definition of $t_k$ from (\ref{1.2}).

\begin{proposition}\label{prop2.2}
Suppose that $Y$ is a real number with $Y\ge X^{1/2}$. Then whenever $s\ge 2t_k$ and $\eps>0$, one 
has
$$\int_0^1|f(\alp)|^s\d\alp \ll Y^{s-1}X^{1-k+\eps}.$$
\end{proposition}

Next, in \S4, we establish an estimate of Weyl-type that delivers non-trivial estimates throughout 
the set of minor arcs $\grm$. Recall here the definition of $K$ given in (\ref{2.0}).

\begin{proposition}\label{prop2.3}
Let $\tet$ be a real number with $\tfrac{5}{6}<\tet<1$, and suppose that $X$ and $Y$ are real numbers 
with $X^\tet\le Y\le X$. Then, whenever $a\in \dbZ$ and $q\in \dbN$ satisfy $(a,q)=1$ and 
$|\alp-a/q|\le q^{-2}$, one has
$$f(\alp)\ll X^\eps Y\left( \Xi^{-1}+X^{-1/2}+\Xi Y^{-2}X^{2-k}\right)^{1/K},$$
where $\Xi=q+Y^2X^{k-2}|q\alp-a|$.
\end{proposition}

We are now equipped to dispose of the minor arc contribution. Suppose that $\alp\in \grm$. By 
Dirichlet's theorem on Diophantine approximation, there exist $a\in \dbZ$ and $q\in \dbN$ with 
$q\le Q$, $(a,q)=1$ and $|q\alp-a|\le Q^{-1}$. The definition of $\grm$ ensures that $q>P$, and thus 
when $k\ge 4$ and $Y=X^\tet$, Proposition \ref{prop2.3} combines with the definition (\ref{2.3}) to 
deliver the bound
$$f(\alp)\ll X^\eps Y(P^{-1}+X^{-1/2}+QY^{-2}X^{2-k})^{1/K}\ll X^{\eps-2\del}Y.$$
Meanwhile, when $k=2$ or $3$, we deduce from the conclusion of \S5 that one has likewise the bound
\begin{equation}\label{2.9}
f(\alp)\ll X^{\eps-2\del}Y.
\end{equation}
We therefore conclude from Proposition \ref{prop2.2} that whenever $s>2t_k$, then for each $\eps>0$, 
one has
\begin{align*}
\rho_s(n;\grm)&\le \Bigl( \sup_{\alp\in \grm}|f(\alp)|\Bigr)^{s-2t_k}\int_\grm |f(\alp)|^{2t_k}\d\alp \\
&\ll (X^{\eps-2\del}Y)^{s-2t_k}Y^{2t_k-1}X^{1-k+\eps}\ll Y^{s-1}X^{1-k-\del}.
\end{align*}
Combining this estimate with (\ref{2.5}) and (\ref{2.7}), we see that under the hypotheses of Theorem 
\ref{theorem1.1}, one has the lower bound $\rho_s(n)\gg Y^{s-1}X^{1-k}$, and thus the proof of 
Theorem \ref{theorem1.1} is complete.\vskip.1cm

\section{Mean value estimates for primes in short intervals} Recall the definition (\ref{2.1}) of the 
exponential sum $f(\alp)$. A heuristic application of the circle method suggests that for all positive integers 
$t$, and for all real numbers $Y$ with $X^{1/2}\le Y\le X$, one should have the bound
\begin{equation}\label{3.1}
\int_0^1|f(\alp)|^{2t}\d\alp \ll Y^t(\log X)^t+Y^{2t-1}X^{1-k}.
\end{equation}
It may be shown using Hua's lemma, meanwhile, that when $2t\ge 2^k$, then
\begin{equation}\label{3.2}
\int_0^1|f(\alp)|^{2t}\d\alp \ll X^\eps Y^{2t-k}.
\end{equation}
It is apparent that when $2t\ge 2^k$ and $Y^{t-1}\ge X^{k-1}$, the bound (\ref{3.1}) is sharper than 
(\ref{3.2}) by a factor $(Y/X)^{k-1}$. It transpires that the methods of Daemen 
\cite{Dae2010a, Dae2010b} permit the proof of a serviceable substitute for (\ref{3.1}) when $t>t_k$, 
and it is this that explains in part the relative success of our approach over that of previous authors.\par

In order to proceed further, we must introduce some additional notation. Let
\begin{equation}\label{3.3}
F(\alp)=\sum_{|m-X|\le Y}e(m^k\alp).
\end{equation}
Also, denote by $J_{s,k}(X)$ the number of integral solutions of the Diophantine system
$$\sum_{i=1}^s(x_i^j-y_i^j)=0\quad (1\le j\le k),$$
with $1\le x_i,y_i\le X$ $(1\le j\le k)$.

\begin{lemma}\label{lemma3.1}
Suppose that $k$ and $t$ are natural numbers with $k\ge 2$. Then whenever $X^{1/3}\le Y\le X$, one 
has
$$\int_0^1|F(\alp)|^{2t}\d\alp \ll (1+Y^2/X)X^{2-k}Y^{\frac{1}{2}k(k+1)-3}J_{t,k}(Y).$$
\end{lemma}

\begin{proof} This is essentially \cite[Theorem 3]{Dae2010a}. The slightly smaller lower bound on 
$Y$ is accommodated by employing precisely the same proof as described in the latter source.
\end{proof}

The latest developments on Vinogradov's mean value theorem (for example, see 
\cite{Woo2012, Woo2013, Woo2014}) supply more powerful estimates for $J_{s,k}(X)$ than were 
available to Daemen.

\begin{lemma}\label{lemma3.2}
When $k\ge 2$ and $t\ge t_k$, one has $J_{t,k}(X)\ll X^{2t-\frac{1}{2}k(k+1)+\eps}$.
\end{lemma}

\begin{proof} The estimate $J_{3,2}(X)\ll X^{3+\eps}$ is immediate from \cite[Theorem 7]{Hua1965}. 
Meanwhile, when $k\ge 3$ and $t\ge k(k-1)$, the bound $J_{t,k}(X)\ll X^{2t-\frac{1}{2}k(k+1)+\eps}$ 
is supplied by \cite[Theorem 1.2]{Woo2014}.
\end{proof}

We now establish Proposition \ref{prop2.2}. Suppose that $Y$ is a real number with $X^{1/2}\le Y\le X$, 
and that $s\ge 2t_k$. Then by applying the trivial estimate $|f(\alp)|=O(Y)$, we find that
\begin{equation}\label{3.4}
\int_0^1|f(\alp)|^s\d\alp \ll Y^{s-2t_k}I(t_k),
\end{equation}
where
$$I(t)=\int_0^1|f(\alp)|^{2t}\d\alp .$$
On recalling (\ref{2.1}), one finds by orthogonality that $I(t)$ counts the number of solutions of the 
equation
$$\sum_{i=1}^t(p_i^k-p_{i+t}^k)=0,$$
with $p_i$ prime and $|p_i-X|\le Y$ $(1\le i\le t)$, in which each solution is counted with weight
$$\prod_{i=1}^{2t}(\log p_i)\ll (\log X)^{2t}.$$
On considering the number of solutions of the underlying Diophantine equation, we therefore see from 
(\ref{3.3}) that
$$I(t)\ll (\log X)^{2t}\int_0^1|F(\alp)|^{2t}\d\alp .$$
We thus deduce from Lemma \ref{lemma3.1} that
$$I(t)\ll (1+Y^2/X)X^{2-k}Y^{\frac{1}{2}k(k+1)-3}J_{t,k}(Y)(\log X)^{2t},$$
whence by Lemma \ref{lemma3.2} and the hypothesis $Y\ge X^{1/2}$, one obtains the bound
$$I(t_k)\ll X^{1-k}Y^{\frac{1}{2}k(k+1)-1}(\log X)^{2t}Y^{2t_k-\frac{1}{2}k(k+1)+\eps}
\ll X^{1-k+2\eps}Y^{2t_k-1}.$$
The conclusion of Proposition \ref{prop2.2} is confirmed by substituting this estimate into (\ref{3.4}).

\section{Estimates of Weyl type, I: $k\ge 4$} Our goal in this section is the proof of Proposition 
\ref{prop2.3}. We begin with an auxiliary lemma concerning exponential sums having a bilinear structure. 
Let $a(m)$ and $b(n)$ be arithmetic functions satisfying the property that for all natural numbers $m$ and 
$n$, one has
\begin{equation}\label{4.1}
a(m)\ll m^\eps\quad \text{and}\quad b(n)\ll n^\eps.
\end{equation}
Let $M$ and $N$ be positive parameters, and define the exponential sum $T(\alp)=T(\alp;M,N)$ by
\begin{equation}\label{4.2}
T(\alp;M,N)=\max_{M'\le 2M}\max_{N'\le 2N}\biggl| \sum_{M<m\le M'}a(m)
\sum_{\substack{N<n\le N'\\ x<mn\le x+y}}b(n)e((mn)^k\alp)\biggr| .
\end{equation}

\begin{lemma}\label{lemma4.1} Let $c>0$ be fixed, and let $x$ and $y$ be positive numbers with 
$x^{3/4}\le y\le x$. Suppose that $M$ and $N$ are positive numbers with $MN\asymp x$ and 
$cx^2/y^2\le N\le c^{-1}y^2/x$. Suppose in addition that $t\ge t_k$ and $2w\ge t_k+2$. Then 
whenever $a\in \dbZ$ and $q\in \dbN$ satisfy $(a,q)=1$ and $|\alp-a/q|\le q^{-2}$, one has
$$T(\alp;M,N)\ll_c x^\eps y\biggl( \frac{1}{q}+\frac{x}{yN^k}+\frac{N^k}{x^{k-1}y}+
\frac{q}{x^{k-2}y^2}\biggr)^{1/(4tw)}.$$  
\end{lemma}

\begin{proof} Throughout this proof, implicit constants may depend on $c$. Let $M'$ and $N'$ be real 
numbers corresponding to the maxima in (\ref{4.2}). Given a $2w$-tuple $\bfn$, write
$$\btil(\bfn)=b(n_1)\ldots b(n_w){\overline{b(n_{w+1})}}\ldots {\overline{b(n_{2w})}}$$
and
$$\sig(\bfn)=\sum_{i=1}^w(n_i^k-n_{w+i}^k).$$
Then in view of the hypothesis (\ref{4.1}), an application of H\"older's inequality leads from 
(\ref{4.2}) to the bound
\begin{equation}\label{4.3}
|T(\alp)|^{2w}\ll x^\eps y^{2w}T_1(\alp),
\end{equation}
where
\begin{equation}\label{4.4}
T_1(\alp)=(M/y)^{2w}M^{-1}\sum_{m\sim M}\sum_\bfn \btil(\bfn)e(m^k\sig(\bfn)\alp ),
\end{equation}
in which the summation is over $2w$-tuples $\bfn$ satisfying
$$N<n_i\le N'\quad \text{and}\quad x/m<n_i\le (x+y)/m\quad (1\le i\le 2w).$$

\par By interchanging the order of summation in (\ref{4.4}), we obtain the bound
\begin{equation}\label{4.5}
T_1(\alp)\ll (M/y)^{2w}M^{-1}\sum_{n_1\sim N}\ldots \sum_{n_{2w}\sim N}|\btil(\bfn)| \biggl| 
\sum_{m\sim M}e(m^k\sig(\bfn)\alp)\biggr| ,
\end{equation}
where the summation over $m$ is subject to the constraint
\begin{equation}\label{4.6}
x/n_i<m\le (x+y)/n_i\quad (1\le i\le 2w).
\end{equation}
Next, invoking symmetry, one discerns that the bound (\ref{4.5}) remains valid if we impose the 
additional condition $n_1\le n_2\le \ldots \le n_{2w}$. The condition (\ref{4.6}) then becomes 
$x/n_1<m\le (x+y)/n_{2w}$, a constraint that implies the relation $n_{2w}<n_1(1+y/x)$. We therefore 
see that, for the $2w$-tuples $\bfn$ at hand, one has
$$|n_{i+1}^k-n_{w+i}^k|\le \left( n_1(1+y/x)\right)^k-n_1^k\le k(n_1y/x)\left( n_1(1+y/x)
\right)^{k-1},$$
whence
$$\biggl| \sum_{i=1}^{w-1}(n_{i+1}^k-n_{w+i}^k)\biggr|\le wk(2n_1)^ky/x.$$

\par Denote by $\kap_{\mu,\nu}(u)$ the number of integral solutions of the equation
$$\sum_{i=1}^{w-1}(n_{i+1}^k-n_{w+i}^k)=u,$$
with
$$\mu\le n_2\le \ldots \le n_{2w-1}\le \nu .$$
Then, again applying (\ref{4.1}), we arrive at the upper bound
\begin{equation}\label{4.7}
T_1(\alp)\ll x^\eps (M/y)^{2w}M^{-1}\sum_{\substack{N<\mu\le \nu \le 2N\\ \nu<\mu (1+y/x)}}
\sum_u \kap_{\mu,\nu}(u)|T_2(\alp;\mu,\nu;u)|,
\end{equation}
where the summation over $u$ is subject to the condition
\begin{equation}\label{4.8}
|u|\le wk(2\mu )^ky/x,
\end{equation}
and
\begin{equation}\label{4.9}
T_2(\alp;\mu,\nu;u)=\sum_{\substack{m\sim M\\ x/\mu<m\le (x+y)/\nu}}e(m^k(\mu^k-\nu^k-u)\alp).
\end{equation}

\par Suppose next that $\nu<\mu (1+y/x)$, consistent with the summation condition in (\ref{4.7}). 
Then another application of H\"older's inequality conveys us from (\ref{4.9}) to the bound
\begin{equation}\label{4.10}
\sum_u \kap_{\mu,\nu}(u)|T_2(\alp;\mu,\nu;u)|\le T_3(\mu,\nu)^{1-1/t}T_4(\mu,\nu)^{1/(2t)}
T_5(\alp;\mu,\nu)^{1/(2t)},
\end{equation}
where
\begin{equation}\label{4.11}
T_3(\mu,\nu)=\sum_u\kap_{\mu,\nu}(u),\quad T_4(\mu,\nu)=\sum_u \kap_{\mu,\nu}(u)^2,
\end{equation}
\begin{equation}\label{4.12}
T_5(\alp;\mu,\nu)=\sum_u|T_2(\alp;\mu,\nu;u)|^{2t},
\end{equation}
and the summations over $u$ are again subject to (\ref{4.8}).\par

We estimate the sums $T_3$, $T_4$ and $T_5$ in turn. Note that our hypothesis $cx^2/y^2\le N$ 
ensures that, in the situation at hand, one has
$$(\mu y/x)^2\mu^{-1}\ge (y^2/x^2)N\ge c,$$
whence $\mu y/x\ge (c\mu)^{1/2}$. Then the definition of 
$\kap_{\mu,\nu}(u)$ takes us from (\ref{4.11}) to the bound
\begin{equation}\label{4.13}
T_3(\mu,\nu)\le (\nu-\mu+1)^{2w-2}\ll (\mu y/x)^{2w-2}.
\end{equation}

Next, when $U$ and $V$ are non-negative numbers, define the Weyl sum
\begin{equation}\label{4.14}
W(\alp;U,V)=\sum_{U\le m\le U+V}e(m^k\alp).
\end{equation}
Since $\mu y/x\ge (c\mu)^{1/2}$, we are led from (\ref{4.11}) via orthogonality and Lemma 
\ref{lemma3.1} to the bound
\begin{align*}
T_4(\mu,\nu)&\le \int_0^1|W(\alp;\mu,\nu-\mu)|^{4w-4}\d\alp \\
&\ll (\mu y/x)^{\frac{1}{2}k(k+1)-1}\mu^{1-k}J_{2w-2,k}(\mu y/x).
\end{align*}
Provided that $N<\mu \le \nu\le 2N$ and $2w-2\ge t_k$, therefore, we deduce from Lemma 
\ref{lemma3.2} that
\begin{equation}\label{4.15}
T_4(\mu,\nu)\ll x^\eps \mu^{1-k}(\mu y/x)^{4w-5}.
\end{equation}

\par Recall that $MN\asymp x$, and that the ambient summation conditions ensure that 
$N\le \mu \le 2N$. Then by substituting (\ref{4.13}) and (\ref{4.15}) into (\ref{4.10}), we deduce 
from (\ref{4.7}) that
\begin{align}
|T_1(\alp)|^{2t}\ll &\, x^\eps (M/y)^{4tw}M^{-2t}(N^2y/x)^{2t}(Ny/x)^{(2w-2)(2t-2)}\notag \\
&\, \times N^{1-k}(Ny/x)^{4w-5}\max_{\mu,\nu}T_5(\alp;\mu,\nu)\notag \\
\ll &\, x^\eps (N^ky/x)^{-1}(y/N)^{-2t}\max_{\mu,\nu}T_5(\alp;\mu,\nu),\label{4.16}
\end{align}
where the maximum over $\mu$ and $\nu$ is subject to the conditions
\begin{equation}\label{4.17}
N<\mu\le \nu\le 2N\quad \text{and}\quad \nu<\mu(1+y/x).
\end{equation}

\par We next turn to the estimation of $T_5(\alp;\mu,\nu)$. Denote by $\ome_{\mu,\nu}(v)$ the number 
of integral solutions of the equation
$$\sum_{i=1}^t(m_i^k-m_{t+i}^k)=v,$$
with
\begin{equation}\label{4.18}
x/\mu<m_i\le (x+y)/\nu \quad \text{and}\quad m_i\sim M\quad (1\le i\le 2t).
\end{equation}
Note that our hypothesis $N\le c^{-1}y^2/x$ combines with (\ref{4.17}) to ensure that, in the situation at 
hand, one has
$$(y/\nu)^2(x/\nu)^{-1}\ge (y^2/x)(2N)^{-1}\ge \tfrac{1}{2}c,$$
whence $y/\nu\ge (\tfrac{1}{2}c)^{1/2}(x/\nu)^{1/2}$. Observe also that 
$\ome_{\mu,\nu}(v)\le \ome_{\nu,\nu}(v)$. Then on recalling the definition (\ref{4.14}), we are led via 
orthogonality and Lemma \ref{lemma3.1} to the bound
\begin{align*}
\ome_{\mu,\nu}(v)&\le \int_0^1 |W(\alp;x/\nu,y/\nu)|^{2t}\d\alp \\
&\ll (y/\nu)^{\frac{1}{2}k(k+1)-1}(x/\nu)^{1-k}J_{t,k}(y/\nu).
\end{align*}
Provided that $N<\mu\le \nu\le 2N$ and $t\ge t_k$, therefore, we deduce from Lemma \ref{lemma3.2} 
that
\begin{equation}\label{4.19}
\ome_{\mu,\nu}(v)\ll (x/\nu)^{1-k}(y/\nu)^{2t-1+\eps}.
\end{equation}

Observe next that the conditions (\ref{4.18}) ensure that
$$\biggl| \sum_{i=1}^t(m_i^k-m_{t+i}^k)\biggr| \le tk(y/\nu)\left( (x+y)/\nu\right)^{k-1}.$$
Then by expanding (\ref{4.12}) and interchanging the order of summation, we find that
$$T_5(\alp;\mu,\nu)\le \sum_v \ome_{\nu,\nu}(v)\biggl| \sum_u
e(v(\mu^k-\nu^k-u)\alp )\biggr|,$$
where the summations over $u$ and $v$ are subject to the conditions (\ref{4.8}) and
$$|v|\le tk(y/\nu)(2x/\nu)^{k-1}.$$
We therefore deduce from (\ref{4.19}) that
\begin{align*}
T_5(\alp;\mu,\nu)&\ll \sum_v\ome_{\nu,\nu}(v)\min \{ N^ky/x,\| v\alp\|^{-1}\}\\
&\ll (x/\nu)^{1-k}(y/\nu)^{2t-1+\eps}\sum_v\min \{ x^{k-2}y^2/v,\|v\alp\|^{-1}\} .
\end{align*}
Thus, as a consequence of \cite[Lemma 2.2]{Vau1997}, we conclude that
\begin{equation}\label{4.21}
(N^ky/x)^{-1}(y/N)^{-2t}\max_{\mu,\nu}T_5(\alp;\mu,\nu)\ll (qx)^\eps \Tet ,
\end{equation}
where, in view of the hypotheses of the lemma concerning $\alp$, one has
$$\Tet \ll \frac{1}{q}+\frac{x}{yN^k}+\frac{N^k}{x^{k-1}y}+\frac{q}{x^{k-2}y^2}.$$
The conclusion of the lemma now follows by substituting (\ref{4.21}) into (\ref{4.16}), and thence 
into (\ref{4.3}).
\end{proof}

In advance of the next lemma, we introduce the positive real number
$$\sig_k=(8t_k)^{-1}\quad (k\ge 2).$$
The following result is an analogue for short intervals of \cite[Lemma 3.2]{KW2001}. We proceed in 
greater generality than is necessary for the application within this paper, since this conclusion is likely to 
find application elsewhere and the additional generality comes at little cost.

\begin{lemma}\label{lemma4.2a}
Let $x$ and $y$ be positive numbers with $1\le y\le x$, and suppose that $\sig$ is a real number with 
$0<\sig\le \sig_k$. Suppose also that $M$ and $N$ are positive numbers with 
$(MN)^{2/3-\sig}\le x^{-1}y^{5/3-\sig}$,
\begin{equation}\label{4.w0}
M^{1-2\sig}N^{2-2\sig}\le x^{1-2/k}y^{2/k-2\sig}\quad \text{and}\quad M^{2-2k\sig}N^{-2k\sig}\le 
x^{1-2/k}y^{2/k-2k\sig}.
\end{equation}
Suppose that $(a_m)$, $(b_n)$ and $(c_n)$ are sequences of complex numbers satisfying 
$|a_m|\le 1+\log m$ and $|b_n|\le 1$ for each $m$ and $n$, and with $c_l=1$ for all $l$, or 
$c_l=\log l$ for all $l$. Suppose further that $\alp$ is a real number, and that there exist $a\in \dbZ$ and 
$q\in \dbN$ satisfying
\begin{equation}\label{4.wa}
(a,q)=1,\quad 1\le q\le (x^{k-2}y^2)^{1/2}\quad \text{and}\quad |q\alp-a|\le (x^{k-2}y^2)^{-1/2}.
\end{equation}
Then one has
\begin{align*}
\sum_{1\le m\le M}a_m\sum_{1\le n\le N}&b_n\sum_{|x-lmn|\le y}c_l e\left((lmn)^k\alp\right) \\
&\ll y^{1-\sig+\eps}(MN)^\sig +\frac{y^{1+\eps}}{\left( q+x^{k-1}y|q\alp-a|\right)^{1/k^2}}.
\end{align*}
\end{lemma}

\begin{proof} Write
$$\Psi(\tet)=\sum_{|l-x/(mn)|\le y/(mn)}c_le(l^k\tet).$$
Put $\calQ=(y/(mn))^{1/3}$, and define $\grW$ to be the union of the arcs
$$\grW(q,a)=\{ \alp\in [0,1):|\tet-a/q|\le \calQ (x/(mn))^{2-k}(y/(mn))^{-2}\},$$
with $0\le a\le q\le \calQ$ and $(a,q)=1$. Also, put $\grw=[0,1)\setminus \grW$. Then it follows from 
the argument underlying the proof of \cite[equation (3.5)]{Dae2010a} that
\begin{equation}\label{4.b}
\sup_{\tet\in \grw}|\Psi(\tet)|\ll (y/(mn))\calQ^{\eps-1/(2t_k)}\ll (y/(mn))^{1-\sig_k}.
\end{equation}
Meanwhile, from \cite[equations (5.1)-(5.5) and \S6]{Dae2010a}, we see that when 
$\tet\in \grW(q,a)\subseteq \grW$, one has
\begin{equation}\label{4.c}
\Psi(\tet)\ll \frac{y/(mn)}{\left( q+(x/(mn))^{k-1}(y/(mn))|q\tet -a|\right)^{1/k}}+\Del,
\end{equation}
where
$$\Del\ll \calQ^{1/2+\eps}\left( 1+\frac{\calQ(x/(mn))^k}{(x/(mn))^{k-2}(y/(mn))^2}\right)^{1/2}
\ll \calQ^{1+\eps}x/y.$$
Provided that $(mn)^{2/3-\sig_k}\le x^{-1}y^{5/3-\sig_k}$, one discerns that
$$\Del\ll (y/(mn))^{1/3+\eps}x/y\le (y/(mn))^{1-\sig_k+\eps}.$$
In combination with (\ref{4.c}), therefore, we conclude that whenever $0<\sig\le \sig_k$, then
\begin{equation}\label{4.d}
\Psi(\tet)\ll (y/(mn))^{1-\sig+\eps}+\frac{y(mn)^{-1}}
{\left( q+(x/(mn))^{k-1}(y/(mn))|q\tet -a|\right)^{1/k}}.
\end{equation}

\par For each integer $m$ with $1\le m\le M$, denote by $\calN$ the set of natural numbers $n$ 
with $1\le n\le N$ for which there exist integers $b$ and $r$ with $(b,r)=1$,
\begin{equation}\label{4.w1}
1\le r\le \tfrac{1}{3}\left( \frac{y}{mn}\right)^{k\sig},\quad |r(mn)^k\alp-b|\le \tfrac{1}{2}
\left( \frac{x}{mn}\right)^{2-k}\left( \frac{y}{mn}\right)^{k\sig-2}.
\end{equation}
Now consider the situation with $\tet=(mn)^k\alp$. By Dirichlet's theorem on Diophantine approximation, 
there exist integers $b$ and $r$ with $(b,r)=1$ and
$$1\le r\le \calQ^{-1}\left( \frac{x}{mn}\right)^{k-2}\left( \frac{y}{mn}\right)^2$$
such that
$$|(mn)^k\alp-b/r|\le \calQ \left( \frac{x}{mn}\right)^{2-k}\left( \frac{y}{mn}\right)^{-2}.$$
If one has $1\le r\le \calQ$, then it follows that $(mn)^k\alp\in \grW$, and then from (\ref{4.d}) one sees 
that
$$\Psi((mn)^k\alp)\ll \left( \frac{y}{mn}\right)^{1-\sig+\eps}+
\frac{y(mn)^{-1}}{\left( r+(x/(mn))^{k-1}(y/(mn))|(mn)^kr\alp -b|\right)^{1/k}}.$$
Observe here that when $n\not\in \calN$, it follows that either $r>\tfrac{1}{3}(y/(mn))^{k\sig}$, or 
else that
$$|r(mn)^k\alp-b|>\tfrac{1}{2}\left( \frac{x}{mn}\right)^{2-k}\left( \frac{y}{mn}\right)^{k\sig-2},$$
whence
$$\Psi((mn)^k\alp)\ll \left( \frac{y}{mn}\right)^{1-\sig+\eps}.$$
If, on the other hand, one has $r>\calQ$, then one discerns that $(mn)^k\alp\in \grw$, and so we deduce 
from (\ref{4.b}) that
$$\Psi((mn)^k\alp)\ll (y/(mn))^{1-\sig+\eps}.$$
\par

The discussion of the previous paragraph supplies the estimate
\begin{equation}\label{4.w2}
\sum_{1\le m\le M}a_m\sum_{1\le n\le N}b_n\sum_{|x-lmn|\le y}c_l e\left((lmn)^k\alp\right) \ll 
E_0+E_1,
\end{equation}
where
$$E_0=\sum_{1\le m\le M}\sum_{1\le n\le N}|a_mb_n|\left( \frac{y}{mn}\right)^{1-\sig+\eps}$$
and
$$E_1=\sum_{1\le m\le M}\sum_{n\in \calN}|a_mb_n|\frac{y(mn)^{-1}\log x}
{\left( r+(x/(mn))^{k-1}(y/(mn))|(mn)^kr\alp -b|\right)^{1/k}}.$$
Here we note that in case $c_l=\log l$ for all $l$, then the desired conclusion follows in like manner from 
the aforementioned work of Daemen by partial summation. In view of our hypotheses concerning $(a_m)$ 
and $(b_n)$, one has
\begin{equation}\label{4.w3}
E_0\ll y^{1-\sig+\eps}(MN)^\sig .
\end{equation}
Also, it is evident from an application of H\"older's inequality that
\begin{equation}\label{4.w4}
E_1\ll (\log x)E_2^{1-1/k}(yE_3)^{1/k},
\end{equation}
where
\begin{equation}\label{4.w5}
E_2=\sum_{1\le m\le M}\sum_{n\in \calN}|a_mb_n|^{k/(k-1)}y/(mn)\ll y^{1+\eps}
\end{equation}
and
\begin{equation}\label{4.w6}
E_3=\sum_{1\le m\le M}\sum_{n\in \calN}\frac{(mn)^{-1}}{r+(x/(mn))^{k-1}(y/(mn))|(mn)^kr\alp -b|}.
\end{equation}

\par For each integer $m$ with $1\le m\le M$, we apply Dirichlet's approximation theorem to deduce the 
existence of $c\in \dbZ$ and $s\in \dbN$ with $(c,s)=1$,
\begin{equation}\label{4.w7}
1\le s\le \left( \frac{x}{mN}\right)^{k-2}\left( \frac{y}{mN}\right)^{2-k\sig}\quad \text{and}
\quad |sm^k\alp-c|\le \left( \frac{x}{mN}\right)^{2-k}\left( \frac{y}{mN}\right)^{k\sig -2}.
\end{equation}
By combining (\ref{4.w1}) and (\ref{4.w7}), we obtain the bound
\begin{align*}
|rn^kc-sb|&\le rn^k\left( \frac{x}{mN}\right)^{2-k}\left( \frac{y}{mN}\right)^{k\sig-2}+
\tfrac{1}{2}s\left( \frac{x}{mn}\right)^{2-k}\left( \frac{y}{mn}\right)^{k\sig -2}\\
&\le \tfrac{1}{2}+\tfrac{1}{3}\left( \frac{x}{m}\right)^{2-k}\left( \frac{y}{m}\right)^{2k\sig-2}
(nN)^{k-k\sig}.
\end{align*}
Then it follows from (\ref{4.w0}) that $|rn^kc-sb|<1$, whence
$$\frac{b}{rn^k}=\frac{c}{s}\quad \text{and}\quad r=\frac{s}{(s,n^k)}.$$
We therefore deduce from (\ref{4.w6}) that
$$E_3=\sum_{1\le m\le M}\frac{m^{-1}}{s+(x/m)^{k-1}(y/m)|m^ks\alp-c|}
\sum_{n\in \calN}\frac{(s,n^k)}{n}.$$
Observe next that
$$\sum_{1\le n\le N}\frac{(s,n^k)}{n}\le \sum_{d|s}
\sum_{\substack{1\le n\le N\\ n^k\equiv 0\mmod{d}}}\frac{d}{n}\ll \sum_{d|s}dv_k(d)^{-1}\log N,$$
in which $v_k(w)$ is the multiplicative function of $w$ defined by taking
$$v_k(p^{uk+v})=p^{u+1},$$
for prime numbers $p$, when $u\ge 0$ and $1\le v\le k$. Hence we deduce that
$$\sum_{1\le n\le N}\frac{(s,n^k)}{n}\ll (\log N)\sum_{d|s}d^{1-1/k}\ll s^{1-1/k+\eps}\log N,$$
whence
\begin{equation}\label{4.w8}
E_3\ll \sum_{1\le m\le M}\frac{s^{\eps-1/k}\log N}{m\left( 1+(x/m)^{k-1}(y/m)|m^k\alp -c/s|\right)}.
\end{equation} 

\par We next define $\calM$ to be the set of natural numbers $m$ with $1\le m\le M$ such that the 
integers $c$ and $s$ defined in (\ref{4.w7}) satisfy
\begin{equation}\label{4.w9}
1\le s\le \tfrac{1}{3}\left( \frac{y}{MN}\right)^{k^2\sig},\quad |sm^k\alp -c|\le \tfrac{1}{3}
\left( \frac{y}{MN}\right)^{k^2\sig}\left( \frac{x}{m}\right)^{1-k}\left( \frac{y}{m}\right)^{-1}.
\end{equation}
In view of (\ref{4.w8}), we find that
\begin{equation}\label{4.w10}
E_3\ll \sum_{m\in \calM}\frac{s^{\eps-1/k}\log N}{m\left( 1+(x/m)^{k-1}(y/m)|m^k\alp-c/s|\right)}
+y^{\eps-k\sig}(MN)^{k\sig}.
\end{equation}
When $a$ and $q$ satisfy (\ref{4.wa}) and $m\in \calM$, it follows from (\ref{4.w9}) that
\begin{align*}
|sm^ka-qc|&\le sm^k(x^{k-2}y^2)^{-1/2}+\tfrac{1}{3}q\left( \frac{y}{MN}\right)^{k^2\sig}
\left( \frac{x}{m}\right)^{1-k}\left( \frac{y}{m}\right)^{-1}\\
&\le \tfrac{1}{3}\left( x^{1-k/2}y^{k^2\sig-1}+x^{-k/2}y^{k^2\sig}\right) M^{k-k^2\sig}
N^{-k^2\sig}.
\end{align*}
The second condition of (\ref{4.w0}) therefore reveals that
$$|sm^ka-qc|<1,$$
whence
$$\frac{c}{sm^k}=\frac{a}{q}\quad \text{and}\quad s=\frac{q}{(q,m^k)}.$$

\par We therefore deduce that
\begin{align*}
\sum_{m\in \calM}&\frac{s^{\eps-1/k}\log N}{m\left( 1+(x/m)^{k-1}(y/m)|m^k\alp-c/s|\right)}\\
&\le \frac{q^\eps \log N}{1+x^{k-1}y|\alp-a/q|}\sum_{1\le m\le M}\frac{(q/(q,m^k))^{-1/k}}{m}.
\end{align*}
In this case, we observe as above that
\begin{align*}
\sum_{1\le m\le M}\frac{(q/(q,m^k))^{-1/k}}{m}&\le \sum_{d|q}
\sum_{\substack{1\le m\le M\\ m^k\equiv 0\mmod{d}}}\frac{(q/d)^{-1/k}}{m}\\
&\ll \sum_{d|q}(q/d)^{-1/k}v_k(d)^{-1}\log M \\
&\ll (\log M)q^{-1/k}\sum_{d|q}1\ll q^{\eps-1/k}\log M.
\end{align*}
We therefore infer that
$$\sum_{m\in \calM}\frac{s^{\eps-1/k}\log N}{m\left( 1+(x/m)^{k-1}(y/m)|m^k\alp-c/s|\right)}
\ll \frac{q^{\eps-1/k}(\log x)^2}{1+x^{k-1}y|\alp-a/q|},$$
whence by (\ref{4.w10}) we find that
$$E_3\ll y^{\eps-k\sig}(MN)^{k\sig}+\frac{(\log x)^2}{(q+x^{k-1}y|q\alp -a|)^{1/k}}.$$

\par On substituting this last estimate together with (\ref{4.w5}) into (\ref{4.w4}), we deduce that
$$E_1\ll y^{1-\sig +\eps}(MN)^\sig +\frac{y^{1+\eps}}{(q+x^{k-1}y|q\alp-a|)^{1/k^2}}.$$
The conclusion of the lemma now follows by substituting this estimate along with (\ref{4.w3}) into 
(\ref{4.w2}).
\end{proof}

We now turn to the problem of estimating the exponential sum $f(\alp)$ defined in (\ref{2.1}). We 
suppose throughout that $\tfrac{5}{6}<\tet\le 1$, $\sig=\sig_k$ and $Y=X^\tet$, and further that 
$a\in \dbZ$ and $q\in \dbN$ satisfy $(a,q)=1$ and $|\alp-a/q|\le q^{-2}$. Let $\Lam(n)$ denote the von 
Mangoldt function, defined by
$$\Lam(n)=\begin{cases} \log p,&\text{when $n=p^l$ for some prime $p$ and natural number $l$,}\\
0,&\text{otherwise,}\end{cases}$$
and let $\mu (n)$ denote the M\"obius function. Suppose that $p$ is a prime number and $l\ge 2$. Then 
whenever $|p^l-X|\le Y$, one has $|p-X^{1/l}|\le YX^{-1+1/l}$, and hence
\begin{equation}\label{4.22}
f(\alp)=\sum_{|n-X|\le Y}\Lam(n)e(n^k\alp)+O(YX^{-1/2}).
\end{equation}

\par We put $U=V=4X^{2-2\tet}$, and apply Vaughan's identity (see \cite{Vau1980}) in the 
shape
$$\Lam(n)=\sum_{\substack{md=n\\ 1\le d\le V}}\mu(d)\log m-\sum_{\substack{lmd=n\\ 1\le d\le V\\ 
1\le m\le U}}\mu(d)\Lam(m)-\sum_{\substack{lmd=n\\ 1\le d\le V\\ m>U\\ ld>V}}\mu(d)\Lam(m).$$
Thus we deduce from (\ref{4.22}) that
$$f(\alp)=S_1-S_2-S_3+O(YX^{-1/2}),$$
where
\begin{align*}
S_1&=\sum_{1\le d\le V}\mu(d)\sum_{|m-X/d|\le Y/d}(\log m)e\left( (md)^k\alp\right) ,\\
S_2&=\sum_{1\le v\le UV}\lam_0(v)\sum_{|l-X/v|\le Y/v}e\left( (lv)^k\alp\right) ,\\
S_3&=\sum_{V<u\le (X+Y)/U}\lam_1(u)\sum_{\substack{|m-X/u|\le Y/u\\ m>U}}\Lam(m)
e\left( (mu)^k\alp\right) ,
\end{align*}
in which
$$\lam_0(v)=\sum_{\substack{md=v\\ 1\le d\le V\\ 1\le m\le U}}\mu(d)\Lam(m)\quad \text{and}\quad 
\lam_1(u)=\sum_{\substack{d|u\\ 1\le d\le V}}\mu(d).$$

\par We begin by estimating the sum $S_3$. Note here that when $\tet>\tfrac{5}{6}$ and $Y=X^\tet$, 
then our choices for $U$ and $V$ ensure that
$$X^2/Y^2<V<(X+Y)/U<Y^2/X.$$
We cover the interval $|n-X|\le Y$ by intervals of the shape $[x,x+y]$ with $|x-X|\le Y$ and 
$x^\tet\le y\le x$. On noting that $|\lam_1(u)|\le \tau(u)$, we find that we may apply Lemma 
\ref{lemma4.1} with $t=t_k$ and $w=\lceil \tfrac{1}{2}(t_k+2)\rceil$. Thus, on interchanging the order 
of summation, we may divide the summation over $u$ into dyadic intervals to deduce that
\begin{align}
S_3&\ll (\log X)\max_{V\le N\le (X+Y)/U}X^\eps Y\left( \frac{1}{q}+\frac{X}{YN^k}+
\frac{N^k}{X^{k-1}Y}+\frac{q}{X^{k-2}Y^2}\right)^{1/K}\notag \\
&\ll X^\eps Y\left( q^{-1}+X^{-1/2}+qX^{2-k}Y^{-2}\right)^{1/K},\label{4.z1}
\end{align}
where $K$ is defined as in (\ref{2.1}).\par

Next we estimate $S_2$. Write
$$S_4(Z,W)=\sum_{Z<v\le W}\lam_0(v)\sum_{|l-X/v|\le Y/v}e((lv)^k\alp).$$
Then we find that
\begin{equation}\label{4.z2}
S_2=S_4(0,V)+S_4(V,UV).
\end{equation}
Note that
$$X^2/Y^2<V<UV<16Y^2/X.$$
In view of the bound $|\lam_0(v)|\le \log v$, we may again divide the summation over $v$ into dyadic 
intervals to deduce from Lemma \ref{lemma4.1} that
\begin{equation}\label{4.z3}
S_4(V,UV)\ll X^\eps Y(q^{-1}+X^{-1/2}+qX^{2-k}Y^{-2})^{1/K}.
\end{equation}
In order to estimate $S_4(0,V)$, we begin by applying Dirichlet's theorem on Diophantine approximation 
to obtain integers $b$ and $r$ with $(b,r)=1$,
$$1\le r\le (X^{k-2}Y^2)^{1/2}\quad \text{and}\quad |r\alp-b|\le (X^{k-2}Y^2)^{-1/2}.$$
Next, noting that
$$V^{2-2k\sig}\le X^{1-2/k}Y^{2/k-2k\sig}\quad \text{and}\quad 
V^{1-2\sig}<X^{1-2/k}Y^{2/k-2\sig},$$
and further that
$$V^{2/3-\sig}\le X^{-1}Y^{5/3-\sig},$$
an application of Lemma \ref{lemma4.2a} with $N=1$ and $b_1=1$ yields
\begin{equation}\label{4.z3a}
S_4(0,V)\ll Y^{1-\sig+\eps}V^\sig +Y^{1+\eps}(r+X^{k-1}Y|r\alp-b|)^{-1/k^2}.
\end{equation}
Note here that if one were to have both $r\le \tfrac{1}{2}q$ and $|r\alp-b|\le q^{-1}Y/X$, then it follows 
from the triangle inequality that
$$qr\left| \frac{a}{q}-\frac{b}{r}\right| \le rq^{-1}+Y/X<1,$$
so that $a/q=b/r$, and indeed $q=r$ and $a=b$. In such circumstances, therefore, one has
\begin{align}
S_4(0,V)&\ll Y^{1-\sig+\eps}V^\sig +Y^{1+\eps}\left( q+X^{k-1}Y|q\alp-a|\right)^{-1/k^2}\notag \\
&\ll Y^{1+\eps}\left( q^{-1}+X^{-1/2}+qX^{2-k}Y^{-2}\right)^{1/K}.\label{4.z4}
\end{align}
Meanwhile, when one has either $r>\tfrac{1}{2}q$ or $|r\alp -b|>q^{-1}Y/X$, the same conclusion 
follows directly from (\ref{4.z3a}). Thus, by combining (\ref{4.z3}) and (\ref{4.z4}), we deduce from 
(\ref{4.z2}) that
\begin{equation}\label{4.z5}
S_2\ll X^\eps Y(q^{-1}+X^{-1/2}+qX^{2-k}Y^{-2})^{1/K}.
\end{equation}

\par Finally, in order to estimate $S_1$, we apply Lemma \ref{lemma4.2a} directly, proceeding as in the 
treatment of $S_4(0,V)$. Thus we again obtain the bound
\begin{equation}\label{4.z6}
S_1\ll X^\eps Y(q^{-1}+X^{-1/2}+qX^{2-k}Y^{-2})^{1/K}.
\end{equation}
Finally, by combining (\ref{4.z1}), (\ref{4.z5}) and (\ref{4.z6}), we obtain the following conclusion.

\begin{lemma}\label{lemma4.2}
Let $\tet$ be a real number with $\tfrac{5}{6}<\tet<1$, and suppose that $X$ and $Y$ are real numbers 
with $X^{\tet}\le Y\le X$. Then whenever $a\in \dbZ$ and $q\in \dbN$ satisfy $(a,q)=1$ and 
$|\alp-a/q|\le q^{-2}$, one has
$$f(\alp)\ll X^\eps Y\left( q^{-1}+X^{-1/2}+qY^{-2}X^{2-k}\right)^{1/K},$$
where $K$ is defined as in {\rm (\ref{2.0})}.
\end{lemma}

Proposition \ref{prop2.3} follows from this lemma by applying a standard transference principle (see 
\cite[Lemma 14.1]{Woo2014a}) to the conclusion of Lemma \ref{lemma4.2}. We note that Kumchev 
\cite[Theorem 1.2]{Kum2013} has stronger conclusions than those available via Lemma \ref{lemma4.2} in 
circumstances wherein $k\ge 3$ and
$$X^{(2k+2)/(2k+3)}<Y\le X.$$
A key feature of Lemma \ref{lemma4.2}, however, is the validity of its estimates for values of $Y$ almost 
as small as $X^{5/6}$.

\section{Estimates of Weyl type, II: $k=2$ and $3$}
Lemma \ref{lemma4.2} would suffice to obtain viable minor arc estimates for $f(\alp)$, but only in 
circumstances wherein $Y>X^{5/6}$. By making use of the earlier literature on the subject, we are able to 
obtain an estimate of the shape (\ref{2.9}) in the cases $k=2$ and $3$ of use when $Y$ is a somewhat 
smaller power of $X$. This we accomplish by dividing the minor arcs $\mathfrak{m}$ into two parts, and 
treating each part in turn. Let
\begin{equation}\label{a.1}
Q_0=X^{1-k}Y^{2k-1}P^{-1},
\end{equation}
and note that $Q_0<Q$. Denote by $\grm_1$ the union of the arcs
$$\{ \alp\in [0,1):|q\alp-a|\le Q_0^{-1}\},$$
with $(a,q)=1$, $0\le a\le q$ and $P<q\le Q_0$, and by $\grm_2$ the union of the arcs
$$\{ \alp\in [0,1):Q^{-1}<|q\alp-a|\le Q_0^{-1}\},$$
with $(a,q)=1$ and $0\le a\le q\le P$. By Dirichlet's theorem on Diophantine approximation, each real 
number $\alpha\in\mathfrak{m}$ can be written in the shape $\alpha=\lam +a/q$, with $(a,q)=1$, 
$0\le a\le q\le Q_0$ and $|q\lambda|<Q_0^{-1}$. In view of the definition of the minor arcs $\grm$, 
we have either
\begin{equation}\label{a.2}
P<q\le Q_0 \quad \text{and}\quad |q\lambda|\le Q_0^{-1},
\end{equation}
or else
\begin{equation}\label{a.3}
1\le q\le P \quad \text{and}\quad Q^{-1}<|q\lambda|\le Q_0^{-1}.
\end{equation}
Thus $\mathfrak{m}=\mathfrak{m}_1\cup\mathfrak{m}_2$.\par

We obtain a bound for $f(\alp)$ when $\alp\in \grm_1$ by means of the following lemma due to 
Tang \cite{Tan2011}.

\begin{lemma}\label{lemmaa.1}
Let $k$ be an integer with $k\ge 2$, and put $\ome=2^{k-1}$. Suppose that $y>x^{1/2}$, and that 
$\alp\in \dbR$ satisfies the property that there exist $a\in \dbZ$ and $q\in \dbN$ with
$$1\le q\le x^{\tfrac{k-1}{\ome-1}}y^{\tfrac{k(\ome-2)+1}{\ome-1}}\quad \text{and}\quad 
|\alp-a/q|\le q^{-2}.$$
Then for any $\eps>0$, one has
$$\sum_{x\le n\le x+y}\Lambda(n)e(n^k\alpha)\ll y^{1+\eps}\Biggl(
\frac{1}{q}+\frac{x^{1/2}}{y}+\frac{x^{\tfrac{\ome^2}{\ome +1}}}{y^\ome}+
\frac{x^{\tfrac{(k-1)(\ome+1)-1}{\ome+1}}}{y^{2k-2}}++\frac{qx^{k-1}}{y^{2k-1}}
\Biggr)^{1/\ome^2}.$$
\end{lemma}

\begin{proof} This is \cite[Lemma 2]{Tan2011}.
\end{proof}

Recall that we may suppose that $Y=X^\tet$ with $\tet_k+4K\del<\tet<1$. Then the definition (\ref{a.1}) 
implies that when $k=2$, we have
$$Q_0=X^{-1}Y^3P^{-1}<XY=X^{\tfrac{k-1}{\ome-1}}Y^{\tfrac{k(\ome-2)+1}{\ome-1}}.$$
Meanwhile, when $k=3$, then instead
$$Q_0=X^{-2}Y^5P^{-1}<X^{2/3}Y^{7/3}=X^{\tfrac{k-1}{\ome-1}}
Y^{\tfrac{k(\ome-2)+1}{\ome-1}}.$$
Then on recalling (\ref{2.3}), (\ref{a.1}) and (\ref{a.2}), we may apply Lemma \ref{lemmaa.1} to show 
that whenever $\alpha\in\mathfrak{m}_1$, then
\begin{equation}\label{a.5}
\sum_{|n-X|\le Y}\Lambda(n)e(n^k\alpha)\ll  X^\eps YP^{-1/\ome^2}\ll YX^{\eps-2\delta}.
\end{equation}

When $\alp\in \mathfrak{m}_2$, meanwhile, we bound $f(\alp)$ via the following estimate of 
Liu, L\"{u} and Zhan \cite{LLZ2006}.

\begin{lemma}\label{lemmaa.2}
Let $k$ be an integer with $k\ge 1$, and suppose that $2\le y\le x$. Let $\alp\in \dbR$, suppose that 
$a\in \dbZ$ and $q\in \dbN$ satisfy $(a,q)=1$, and define
$$\Xi=x^k|\alp-a/q|+x^2y^{-2}.$$
Then for any $\eps>0$, one has
\begin{align*}
\sum_{x<n\le x+y}&\Lambda(n)e(n^k\alpha)\\
&\ll (qx)^\eps\left(\frac{y(q\Xi)^{1/2}}{x^{1/2}}
+(qx)^{1/2}\Xi^{1/6}+y^{1/2}x^{3/10}+\frac{x^{4/5}}{\Xi^{1/6}}+\frac{x}{(q\Xi)^{1/2}}\right).
\end{align*}
\end{lemma}

\begin{proof} This is \cite[Theorem 1.1]{LLZ2006}.
\end{proof}

We again recall that $Y=X^\tet$ with $\tet_k+4K\del<\tet<1$. Consequently, under the hypotheses of 
Lemma \ref{lemmaa.2}, when $\alpha\in \mathfrak{m}_2$, we have $\Xi>X^2Y^{-2}$ and
$$X^kQ^{-1}<q\Xi=|q\alp-a|X^k+qX^2Y^{-2}\le X^kQ_0^{-1}+PX^2Y^{-2}<X^kQ_0^{-1}.$$
Thus on recalling (\ref{2.3}), (\ref{a.1}) and (\ref{a.2}), we find via Lemma \ref{lemmaa.2} that 
whenever $\alpha\in\mathfrak{m}_2$, then
\begin{align*}
\sum_{|n-X|\le Y}\Lambda(n)e(n^k\alpha)\ll X^{\eps}&\left( X^{(k-1)/2}YQ_0^{-1/2}
+P^{1/3}X^{(k+3)/6}Q_0^{-1/6}\right.\\
&\, \ \ \left.+Y^{1/2}X^{3/10}+Y^{1/3}X^{7/15}+Q^{1/2}X^{1-k/2}\right),
\end{align*}
whence the estimate (\ref{a.5}) follows also when $\alp\in \grm_2$.
Since $\grm=\grm_1\cup \grm_2$, it follows from (\ref{a.5}) that when $k=2$ and $3$, and 
$\alpha\in \mathfrak{m}$, one has
$$\sum_{|n-X|\le Y}\Lambda(n)e(n^k\alpha)\ll YX^{\eps-2\delta}.$$
This combined with (\ref{4.22}) delivers the bound (\ref{2.9}) for $k=2,3$.

\section{The major arc analysis: preliminaries} We analyse the major arc contribution by applying the 
iterative idea of \cite{Liu2005}. Recall the definition (\ref{2.3}) of $P$ and $Q$, and write
$$N_1=X-Y\quad \text{and}\quad N_2=X+Y.$$
Suppose that $a\in \dbZ$, $q\in \dbN$ and $\lam\in \dbR$ satisfy $(a,q)=1$ and $|\lam|\le (qQ)^{-1}$, 
and consider the value of $f(\alp)$ when $\alp=\lam+a/q$.\par

Write $\chi$ for a typical Dirichlet character modulo $q$, and denote the principal character by 
$\chi_0$. Also, let $\del_\chi$ be $1$ or $0$ according to whether $\chi$ is principal or not. Observe 
that when $1\le q\le P$ and $p$ is a prime number with $N_1\le p\le N_2$, then $(q,p)=1$. Then we 
may rewrite $f(\lam+a/q)$ in the form
\begin{align*}
f(\lam+a/q)&=\sum_{\substack{N_1\le p\le N_2\\ (p,q)=1}}(\log p)e\left( p^k(\lam+a/q)\right) \\
&=\phi(q)^{-1}C(q,a)V(\lam)+\phi(q)^{-1}\sum_{\chi\nmod q}C(\chi,a)W(\chi,\lam),
\end{align*}
where
\begin{align*}
C(\chi,a)&=\sum_{h=1}^q\chibar(h)e(h^ka/q),\quad C(q,a)=C(\chi_0,a),\\
V(\lam)&=\sum_{N_1\le m\le N_2}e(m^k\lam),\\
W(\chi,\lam)&=\sum_{N_1\le p\le N_2}(\log p)\chi(p)e(p^k\lam)-\del_\chi V(\lam).
\end{align*}
Thus we discern that
\begin{equation}\label{5.1}
\int_\grM f(\alp)^se(-n\alp)\d\alp =\sum_{j=0}^s\binom{s}{j}I_j,
\end{equation}
where
\begin{equation}\label{5.2}
I_j=\sum_{1\le q\le P}\sum^q_{\substack{a=1\\ (a,q)=1}}\phi(q)^{-s}C(q,a)^{s-j}e(-na/q)\calI_j(q,a),
\end{equation}
in which
\begin{equation}\label{5.3}
\calI_j(q,a)=\int_{-1/(qQ)}^{1/(qQ)}V(\lam)^{s-j}\biggl( \sum_{\chi\nmod{q}}C(\chi,a)
W(\chi,\lam)\biggr)^je(-n\lam)\d\lam .
\end{equation}

\par We shall find that $I_0$ provides the main contribution on the right hand side of (\ref{5.1}), while 
$I_1,I_2,\ldots ,I_s$ contribute to an error term. We begin by computing the main term $I_0$. Define
\begin{align*}
B(n,q;\chi_1,\ldots ,\chi_s)&=\sum^q_{\substack{a=1\\ (a,q)=1}}C(\chi_1,a)\ldots C(\chi_s,a)e(-na/q),\\
B(n,q)&=B(n,q;\chi_0,\ldots ,\chi_0),
\end{align*}
and
$$\grS(n)=\sum_{q=1}^\infty \phi(q)^{-s}B(n,q).$$
We observe that $\grS(n)$  is the usual singular series in the Waring-Goldbach problem defined in \S2. 
Thus, whenever $s>\max\{3,k(k-1)\}$ and $n\equiv s\mmod{R(k)}$ (and, in the case $k=3$ and $s=7$, 
one has in addition $9\nmid n$), there is a positive number $\eta=\eta(s,k)$ for which 
$1\ll \grS(n)\ll (\log X)^\eta$. An auxiliary estimate facilitates our transition from truncated singular series 
to the completed singular series $\grS(n)$.\par

\begin{lemma}\label{lemma5.1}
Let $q$ and $r_1,\ldots r_s$ be natural numbers, and denote by $r_0$ the least common multiple 
$[r_1,\ldots r_s]$. Let $\chi_j\nmod{r_j}$ be primitive characters for $1\le j\le s$, and write 
$\chi_0\nmod{q}$ for the principal character. Then there is a positive number $c$, independent of 
$q$ and $r_1,\ldots ,r_s$, such that
$$\sum_{\substack{1\le q\le x\\ r_0|q}}\phi(q)^{-s}|B(n,q;\chi_1\chi_0,\ldots ,\chi_s\chi_0)|\ll 
r_0^{1+\eps-s/2}L^c.$$
\end{lemma}

\begin{proof} The desired conclusion follows by means of the argument of the proof of 
\cite[Lemma 6.7(a)]{LL1993}.
\end{proof}

We now examine the expression
\begin{equation}\label{5.5}
I_0=\sum_{1\le q\le P}\phi(q)^{-s}B(n,q)\int_{-1/(qQ)}^{1/(qQ)}V(\lam)^se(-n\lam)\d\lam .
\end{equation}
Write
$$V^*(\lam)=k^{-1}\sum_{N_1^k\le m\le N_2^k}m^{-1+1/k}e(\lam m)$$
and
$$\calJ(n)=\int_{-1/2}^{1/2}V^*(\lam)^se(-n\lam)\d\lam .$$
We remark that, by orthogonality, one finds that $\calJ(n)$ counts the number of solutions of the 
equation
$$m_1+\ldots +m_s=n,$$
with $N_1^k\le m_i\le N_2^k$ $(1\le i\le s)$, and with each solution $\bfm$ counted with weight 
$k^{-s}(m_1\ldots m_s)^{-1+1/k}$. Thus we have
\begin{equation}\label{5.6}
\calJ(n)\asymp Y^{s-1}X^{1-k}.
\end{equation}
Our first step in the analysis of the relation (\ref{5.5}) is to replace the integral on the right 
hand side by $\calJ(n)$, with an acceptable error term.\par

First, by partial summation and a change of variable (compare the discussion following the proof of 
\cite[Lemma 2.7]{Vau1997}), one finds that
\begin{align}
V(\lam)&=\int_{N_1}^{N_2}e(\lam u^k)\d u+O(1+X^{k-1}Y|\lam|)\notag \\
&=V^*(\lam)+O(1+X^{k-1}Y|\lam|).\label{5.7}
\end{align}
By partial summation, one obtains the estimate
\begin{equation}\label{5.8}
V^*(\lam)\ll Y(1+X^{k-1}Y\|\lam \|)^{-1}.
\end{equation}
Recall from (\ref{2.3}) that $Q=X^{k-2}Y^2P^{-1}$, and recall also that our hypotheses concerning $Y$ 
ensure that $Y\ge X^{3/4}$. Thus $X^{k-1}Y/Q=XP/Y\ll Y^{1/2}$, and we see from (\ref{5.7}) that
whenever $|\lam|\le 1/(qQ)$, then $V(\lam)=V^*(\lam)+O(Y^{1/2})$. These estimates lead us, via 
(\ref{5.8}), to the relation
\begin{align*}
\int_{-1/(qQ)}^{1/(qQ)}V(\lam)^s&e(-n\lam)\d\lam -\int_{-1/(qQ)}^{1/(qQ)}V^*(\lam)^s
e(-n\lam)\d\lam \\
&\ll Y^{1/2}\int_{-1/(qQ)}^{1/(qQ)}Y^{s-1}(1+X^{k-1}Y|\lam|)^{1-s}\d\lam +Y^{s/2}Q^{-1}\\
&\ll Y^{s-3/2}X^{1-k}+PY^{-2+s/2}X^{2-k}.
\end{align*}
Thus, when $s\ge 3$, we see from (\ref{5.8}) that when $1\le q\le P$, one has
\begin{align*}
\int_{-1/(qQ)}^{1/(qQ)}V(\lam)^se(-n\lam)\d\lam -\calJ(n)&\ll \int_{1/(qQ)}^{1/2}\frac{Y^s}
{(1+X^{k-1}Y\lam)^s}\d\lam +Y^{s-3/2}X^{1-k}\\
&\ll Y^{s-1}X^{1-k}\left(\frac{X^{k-1}Y}{PQ}\right)^{-1}+Y^{s-3/2}X^{1-k}\\
&\ll Y^sX^{-k}+Y^{s-3/2}X^{1-k}.
\end{align*}
Hence we deduce that
\begin{equation}\label{5.9}
\int_{-1/(qQ)}^{1/(qQ)}V(\lam)^se(-n\lam)\d\lam =\calJ(n)+O(Y^{s-1}X^{1-k-2\del}).
\end{equation}

\par Recall next from \cite[Lemma 8.5]{Hua1965} that when $(a,q)=1$, one has
$$C(\chi_0,a)=\sum^q_{\substack{h=1\\ (h,q)=1}}e(ah^k/q)\ll q^{1/2+\eps}.$$
Thus we deduce that whenever $s\ge 5$, one has
$$\sum_{q>P}\phi(q)^{-s}B(n,q)\ll \sum_{q>P}\phi(q)^{1-s}(q^{1/2+\eps})^s\ll P^{\eps-1/2}$$
and
$$\sum_{1\le q\le P}\phi(q)^{-s}|B(n,q)|\ll 1.$$
By substituting (\ref{5.9}) into (\ref{5.5}), we therefore deduce that
\begin{align*}
I_0-\calJ(n)\sum_{1\le q\le P}\phi(q)^{-s}B(n,q)&\ll Y^{s-1}X^{1-k-2\del}
\sum_{1\le q\le P}\phi(q)^{-s}|B(n,q)|\\
&\ll Y^{s-1}X^{1-k-\del},
\end{align*}
whence, in view of (\ref{5.6}), one sees that
$$I_0-\grS(n)\calJ(n)\ll P^{\eps-1/2}Y^{s-1}X^{1-k}+Y^{s-1}X^{1-k-\del}.$$
We thus conclude that
\begin{equation}\label{5.10}
I_0=\grS(n)\calJ(n)+O(Y^{s-1}X^{1-k-\del}).
\end{equation}
In particular, provided that $s>\max\{4,k(k-1)\}$ and $n\equiv s\mmod{R(k)}$ (and, in the case $k=3$ 
and $s=7$, one has in addition $9\nmid n$), then
\begin{equation}\label{5.11}
I_0\gg Y^{s-1}X^{1-k}.
\end{equation}

\par It remains to handle the contribution of the remaining terms $I_j$ within (\ref{5.1}). The 
discussion of these terms requires us to analyse the auxiliary expressions
$$J_\nu (g)=\sum_{1\le r\le P}[g,r]^{1+\nu-s/2}\sideset{}{^*}\sum_{\chi\nmod{r}}
\max_{|\lam|\le 1/(rQ)}|W(\chi,\lam)|$$
and
$$K_\nu (g)=\sum_{1\le r\le P}[g,r]^{1+\nu-s/2}\sideset{}{^*}\sum_{\chi\nmod{r}}
\biggl( \int_{-1/(rQ)}^{1/(rQ)}|W(\chi,\lam)|^2\d\lam \biggr)^{1/2},$$
in which the asterisk decorating the summations over characters $\chi$ is used to denote that the sums 
are restricted to primitive characters modulo $r$. Throughout, we define $P$ and $Q$ as in (\ref{2.3}). 
In \S7 we establish the following estimate for $K_\nu(g)$.

\begin{lemma}\label{lemma5.2}
Let $\nu$ be a sufficiently small positive number. Then there is a positive number $c$ with the property 
that
$$K_\nu(g)\ll g^{1+2\nu-s/2}(YX^{1-k})^{1/2}L^c.$$
\end{lemma}

In \S8 we turn our attention toward the expression $J_\nu(g)$, poviding the following estimates.

\begin{lemma}\label{lemma5.3}
Let $\nu$ be a sufficiently small positive number. Then there is a positive number $c$ with the 
property that
$$J_\nu(g)\ll g^{1+2\nu-s/2}YL^c.$$
Also, for any positive number $B$, one has
$$J_\nu(1)\ll YL^{-B}.$$
\end{lemma}

Granted the validity of Lemmata \ref{lemma5.2} and \ref{lemma5.3}, we are now equipped to establish 
that for each positive number $A$, one has
\begin{equation}\label{5.14}
I_j\ll Y^{s-1}X^{1-k}L^{-A}.
\end{equation}
By substituting this estimate together with (\ref{5.10}) and (\ref{5.11}) into (\ref{5.1}), we 
conclude that
$$\int_\grM f(\alp)^se(-n\alp)\d\alp =\grS(n)\grJ(n)+O(Y^{s-1}X^{1-k}L^{-1})\gg Y^{s-1}X^{1-k},$$
thereby completing the proof of Proposition \ref{prop2.1}.\par

We begin by confirming the estimate (\ref{5.14}) in the case $j=s$. By reference to (\ref{5.2}) and 
(\ref{5.3}), we may reduce the characters occurring in the summation into primitive characters, 
thereby obtaining the upper bound
$$|I_s|\le \sum_{\bfr}\sideset{}{^*}\sum_{\bfchi\nmod{\bfr}}\sum_{\substack{1\le q\le P\\ r_0|q}}
\phi(q)^{-s}|B(n,q;\chi_1\chi_0,\ldots ,\chi_s\chi_0)|\calI(r_0,\bfchi),$$
where the first summation denotes a sum over $1\le r_j\le P$ $(1\le j\le s)$, the second denotes 
one over $\chi_j\nmod{r_j}$ $(1\le j\le s)$, and
$$\calI(r_0,\bfchi)=\int_{-1/(r_0Q)}^{1/(r_0Q)}|W(\chi_1,\lam)\ldots W(\chi_s,\lam)|\d\lam .$$
An application of Lemma \ref{lemma5.1} yields the bound
$$I_s\ll L^c\sum_\bfr r_0^{1+\eps-s/2}\sideset{}{^*}\sum_{\bfchi \nmod{\bfr}}\calI(r_0,\bfchi).$$
By Cauchy's inequality, therefore, we discern that
\begin{equation}\label{5.15}
I_s\ll L^c\sum_{\bfr'}\sideset{}{^*}\sum_{\bfchi'\nmod{\bfr'}}\biggl(\prod_{j=1}^{s-2}
\max_{|\lam|\le 1/(r_jQ)}|W(\chi_j,\lam)|\biggr)\grT(r_1,\ldots,r_{s-2}),
\end{equation}
where the first summation denotes a sum over $1\le r_j\le P$ $(1\le j\le s-2)$, the second denotes one 
over $\chi_j\nmod{r_j}$ $(1\le j\le s-2)$, and
\begin{equation}\label{5.16}
\grT(r_1,\ldots ,r_{s-2})=\sum_{1\le r_{s-1}\le P}\ \sideset{}{^*}\sum_{\chi_{s-1}\nmod{r_{s-1}}}
\grT_1(r_{s-1},\chi_{s-1})\grT_2(r_1,\ldots ,r_{s-1}),
\end{equation}
in which
$$\grT_1(r_{s-1},\chi_{s-1})=\biggl( \int_{-1/(r_{s-1}Q)}^{1/(r_{s-1}Q)}|W(\chi_{s-1},\lam)|^2\d\lam 
\biggr)^{1/2}$$
and
$$\grT_2(r_1,\ldots ,r_{s-1})=\sum_{1\le r_s\le P}r_0^{1+\eps-s/2}
\sideset{}{^*}\sum_{\chi_s\nmod{r_s}}\biggl( \int_{-1/(r_sQ)}^{1/(r_sQ)}|W(\chi_s,\lam)|^2\d\lam 
\biggr)^{1/2}.$$

\par Now we follow the iterative procedure of \cite{Liu2005} in order to bound the above sums over 
$r_s, r_{s-1},\dots ,r_1$ in turn. First, on noting that
$$r_0=[r_1,\ldots ,r_s]=[[r_1,\ldots ,r_{s-1}],r_s],$$
we deduce from Lemma \ref{lemma5.2} that
$$\grT_2(r_1,\ldots ,r_{s-1})=K_\eps ([r_1,\ldots ,r_{s-1}])\ll 
[r_1,\ldots ,r_{s-1}]^{1+2\eps-s/2}Y^{1/2}X^{(1-k)/2}L^c.$$
Substituting this estimate into (\ref{5.16}), and applying Lemma \ref{lemma5.2} once again, we 
deduce that
\begin{align*}
\grT(r_1,\ldots ,r_{s-2})&\ll Y^{1/2}X^{(1-k)/2}L^cK_{2\eps}([r_1,\ldots ,r_{s-2}])\\
&\ll [r_1,\ldots ,r_{s-2}]^{1+4\eps-s/2}YX^{1-k}L^{2c}.
\end{align*}
Next substituting this estimate into (\ref{5.15}), we may apply Lemma \ref{lemma5.3} to bound the sum 
over $r_{s-2}$ and $\chi_{s-2}$ on the right hand side to obtain the bound
\begin{align*}\sum_{r_{s-2}}\sum_{\chi_{s-2}}\max_{|\lam|\le 1/(r_{s-2}Q)}
&|W(\chi_{s-2},\lam)|\grT(r_1,\ldots ,r_{s-2})\\
&\ll YX^{1-k}L^{2c}J_{4\eps }([r_1,\ldots ,r_{s-3}])\\
&\ll [r_1,\ldots ,r_{s-3}]^{1+8\eps-s/2}Y^2X^{1-k}L^{3c}.
\end{align*}
Proceeding in a similar manner to bound successively the summations over $r_{s-3},\ldots,r_1$, we 
arrive in the final step at the bound
$$I_s\ll Y^{s-2}X^{1-k}L^{(s-1)c}\sum_{1\le r_1\le P}r_1^{1+2^{s-1}\eps-s/2}\sideset{}{^*}
\sum_{\chi_1\nmod{r_1}}\max_{|\lam|\le 1/(r_1Q)}|W(\chi_1,\lam)|.$$
We therefore conclude from an application of Lemma \ref{lemma5.3}, with $B$ taken to be sufficiently 
large in terms of $c$, that for each positive number $A$ one has
$$I_s\ll Y^{s-2}X^{1-k}L^{(s-1)c}J_{2^{s-1}\eps }(1)\ll Y^{s-1}X^{1-k}L^{-A}.$$
This confirms the estimate (\ref{5.14}) in the case $j=s$.\par

The other terms $I_j$ with $1\le j\le s-1$ may be handled in a manner similar to, though simpler than, 
that applied in the case $j=s$. Recall (\ref{2.3}) and the hypothesis $Y\ge X^{3/4}$. Then one observes 
that, as a consequence of (\ref{5.7}) and (\ref{5.8}), when $|\lam|\le 1/Q$, one has 
$$V(\lam)=k^{-1}\sum_{N_1^k\le m\le N_2^k}m^{-1+1/k}e(\lam m)+O(Y^{1/2}),\quad 
\max_{|\lam|\le Q^{-1}}|V(\lam)|\ll Y,$$
and
\begin{align*}
\int_{-1/Q}^{1/Q}|V(\lam)|^2\d\lam &\ll Y^2(X^{k-1}Y)^{-1}+YQ^{-1}\\
&\ll YX^{1-k}+PY^{-1}X^{2-k}\ll YX^{1-k}.
\end{align*}
This completes our account of the proof of the estimates (\ref{5.14}), and 
hence also of the proof of Proposition \ref{prop2.1}, subject to the account in the following two 
sections of the proof of Lemmata \ref{lemma5.2} and \ref{lemma5.3}.

\section{The estimation of $K_\nu (g)$} The approach that we adopt in the proofs of Lemmata 
\ref{lemma5.2} and \ref{lemma5.3} is similar to that of \cite[\S2]{LW2008}, which combines the standard 
approach to such problems (as in \cite{LX2007} and \cite{ST2009}) and the mean value theorem of Choi 
and Kumchev \cite{CK2006}. We begin by introducing a special case of an immediate generalisation of 
\cite[Lemma 2.2]{LW2008}.

\begin{lemma}\label{lemma6.1} Let $\chi$ be a Dirichlet character modulo $r$. In addition, let $\calQ$, 
$\calX$ and $\calY$ be real numbers with $\calQ\ge r$, $2\le \calX<\calY\le 2\calX$ and 
$\|\calX\|\asymp \|\calY\|\asymp 1$, and put
$$\calT_0=\left( \log (\calY/\calX)\right)^{-1},\quad \calT_1=\min \left\{ 
\left( \log (\calY/\calX)\right)^{-2}, 4k\pi \calX^k/(r\calQ)\right\},$$
$$\calT_2=4k\pi \calX^k/(r\calQ),\quad \calT_3=\calX^{2k}\quad \text{and}\quad 
\kap=(\log \calX)^{-1}.$$
Define $F(s)=F(s,\chi)$ by
$$F(s,\chi)=\sum_{\calX<n\le 2\calX}\Lam(n)\chi(n)n^{-s}$$
and $W(\bet)=W_\chi(\bet;\calX,\calY)$ by
$$W_\chi(\bet;\calX,\calY)=\biggl|\sum_{\calX\le n\le \calY}\Lam(n)\chi(n)e(\bet n^k)\biggr|.$$
Then we have
\begin{align*}
\max_{|\bet|\le 1/(r\calQ)}W(\bet)\ll &\, \log (\calX/\calY)\underset{|\tau|\le \calT_1}{\int}
|F(\kap+i\tau)|\d\tau +\underset{\calT_1<|\tau|\le \calT_2}{\int}\frac{|F(\kap+i\tau)|}{|\tau|^{1/2}}
\d\tau \\
&\, +\underset{\calT_2<|\tau|\le \calT_3}{\int}\frac{|F(\kap+i\tau)|}{|\tau|}\d\tau +1
\end{align*}
and
$$W(0)\ll \log (\calX/\calY)\underset{|\tau|\le \calT_0}{\int}|F(\kap+i\tau)|\d\tau +
\underset{\calT_0<|\tau|\le \calT_3}{\int}\frac{|F(\kap+i\tau)|}{|\tau|}\d\tau +1.$$
\end{lemma}

We next record the mean value theorem of Choi and Kumchev in the modified form presented by Li and 
Wu (see the formulation of \cite[Theorem 1.1]{CK2006} given in \cite[Lemma 2.1]{LW2008}).

\begin{lemma}\label{lemma6.2}
Let $l\in \dbN$, and let $R$, $T$ and $\calX$ be real numbers with $R\ge 1$, $T\ge 1$ and $\calX\ge 1$. 
Finally, put $\kap=(\log \calX)^{-1}$. Then there is an absolute constant $c>0$ for which
$$\sum_{\substack{r\sim R\\ l|r}}\,\sideset{}{^*}\sum_{\chi\nmod{r}}\int_{-T}^T\biggl| 
\sum_{\calX\le n\le 2\calX}\frac{\Lam(n)\chi(n)}{n^{\kap+i\tau}}\biggr| \d\tau \ll (l^{-1}R^2T
\calX^{11/20}+\calX)\left( \log (RT\calX)\right)^c.$$
\end{lemma}

We are now equipped to estimate $K_\nu (g)$, thus completing the proof of Lemma \ref{lemma5.2}. Let 
$\nu$ be a sufficiently small positive number. Write
$$\What(\chi,\lam)=\sum_{N_1\le m\le N_2}\left( \Lam(m)\chi(m)-\del_\chi\right) e(m^k\lam),$$
and then put
$$\Ups=\sum_{1\le r\le P}[g,r]^{1+\nu-s/2}\sideset{}{^*}\sum_{\chi\nmod{r}}
\biggl( \int_{-1/(rQ)}^{1/(rQ)}
|W(\chi,\lam)-\What(\chi,\lam)|^2\d\lam \biggr)^{1/2}.$$
Next, observe that when $X$ is large, the condition $|p^j-X|\le Y$ implies that 
$|p-X^{1/j}|\le YX^{-1+1/j}$, and hence for all $\lam\in\dbR$, one has
$$\What(\chi,\lam)-W(\chi,\lam)=\sum_{j\ge 2}\sum_{N_1\le p^j\le N_2}(\log p)\chi(p^j)e(p^{jk}\lam)
\ll YX^{-1/2}.$$
Then on making use of the relation $[g,r](g,r)=gr$, we deduce that
\begin{align*}
\Ups&\ll YX^{-1/2}\sum_{1\le r\le P}[g,r]^{1+\nu-s/2}(r/Q)^{1/2}\\
&\ll YX^{-1/2}g^{1+\nu-s/2}Q^{-1/2}\sum_{1\le r\le P}\left( \frac{r}{(g,r)}
\right)^{1+\nu-s/2}r^{1/2}.
\end{align*}
Consequently, recalling the definition (\ref{2.3}) of $P$ and $Q$, and noting that we are permitted to 
assume that $s\ge 5$, we see that
\begin{align*}
\Ups&\ll YX^{-1/2}g^{1+\nu-s/2}Q^{-1/2}P^{1/2}\sum_{\substack{1\le d\le P\\ d|g}}
\sum_{1\le \rho\le P/d}\rho^{1+\nu-s/2}\\
&\ll YX^{-1/2}g^{1+2\nu-s/2}(X^{k-2}Y^2P^{-1})^{-1/2}P^{1/2+\eps}.
\end{align*}
Thus we conclude that $\Ups\ll g^{1+2\nu-s/2}(YX^{1-k})^{1/2}$, so that in order to establish 
Lemma \ref{lemma5.2}, it suffices to confirm that whenever $\frac{1}{2}\le R\le P$, one has
\begin{equation}\label{6.1}
\sum_{r\sim R}[g,r]^{1+\nu-s/2}\sideset{}{^*}\sum_{\chi\nmod{r}}\biggl( \int_{-1/(rQ)}^{1/(rQ)}
|\What (\chi,\lam)|^2\d\lam \biggr)^{1/2}\ll g^{1+2\nu-s/2}(YX^{1-k})^{1/2}L^c.
\end{equation}

\par We observe next that an application of Gallagher's lemma (see \cite[Lemma 1]{Gal1970}) shows 
that whenever $r\sim R$, then
\begin{align}
\int_{-1/(rQ)}^{1/(rQ)}|\What (\chi,\lam)|^2\d\lam &\ll (RQ)^{-2}\int_{-\infty}^\infty 
\biggl| \sum_{\substack{N_1^k\le m^k\le N_2^k\\ |m^k-v|\le RQ/3}}\left( 
\Lam (m)\chi(m)-\del_\chi\right) \biggr|^2\d v \notag \\
&\ll (RQ)^{-2}\int_{N_1^k-RQ/3}^{N_2^k+RQ/3}\biggl| \sum_{U\le m\le V}\left( 
\Lam(m)\chi(m)-\del_\chi\right) \biggr|^2\d v,\label{6.2}
\end{align}
where we write
$$U=\max\{ N_1,(v-RQ/3)^{1/k}\}\quad \text{and}\quad V=\min \{ N_2,(v+RQ/3)^{1/k}\} .$$

\par We begin by examining the situation with $R=\frac{1}{2}$ and $r=1$ . Here we have
\begin{align*}
\biggl| \sum_{U\le m\le V}\left( \Lam(m)\chi(m)-\del_\chi\right) \biggr|&=
\biggl|\sum_{U\le m\le V}(\Lam(m)-1)\biggr| \\
&\ll \left( (v+Q/3)^{1/k}-(v-Q/3)^{1/k}\right) L\\
&\ll X^{1-k}QL.
\end{align*}
On substituting this conclusion into (\ref{6.2}), we deduce that
$$\int_{-1/Q}^{1/Q}|\What (\chi_0,\lam)|^2\d\lam \ll Q^{-2}(X^{k-1}Y+Q)(X^{1-k}QL)^2,$$
whence
$$g^{1+\nu-s/2}\biggl( \int_{-1/Q}^{1/Q}|\What(\chi_0,\lam)|^2\d\lam \biggr)^{1/2}\ll 
g^{1+\nu-s/2}
(YX^{1-k})^{1/2}L.$$
This confirms (\ref{6.1}) in the case $R=\frac{1}{2}$.\par

Suppose next that $R\ge 1$ and $r\sim R$. In these circumstances, we have $\del_\chi=0$. We apply 
Lemma \ref{lemma6.1} to estimate the integral on the right hand side of (\ref{6.2}), taking $\calX=U$ 
and $\calY=V$, and making use of the notation of the statement of that lemma. We observe in this 
context that
$$(N_2^k+RQ/3)-(N_1^k-RQ/3)\ll X^{k-1}Y.$$
Define $T_i=T_i(\chi,r)$ for $i=1$ and $2$ by putting
$$T_1(\chi,r)=\calT_0^{-1}\underset{|\tau|\le \calT_0}\int |F(\kap+i\tau)|\d\tau\quad \text{and}
\quad T_2(\chi,r)=\underset{\calT_0<|\tau|\le \calT_3}\int \frac{|F(\kap+i\tau)|}{|\tau|}\d\tau .$$
Then we infer from Lemma \ref{lemma6.1} that
$$\biggl( \int_{N_1^k-RQ/3}^{N_2^k+RQ/3}\biggl| \sum_{U\le m\le V}\Lam(m)\chi(m)\biggr|^2\d v
\biggr)^{1/2}\ll (X^{k-1}Y)^{1/2}(T_1+T_2+1).$$
By substituting this conclusion first into (\ref{6.2}), and then into (\ref{6.1}), we conclude thus far that
\begin{equation}\label{6.3}
\sum_{r\sim R}[g,r]^{1+\nu-s/2}\sideset{}{^*}\sum_{\chi\nmod{r}}\biggl( \int_{-1/(rQ)}^{1/(rQ)}
|\What (\chi,\lam)|^2\d\lam \biggr)^{1/2}\ll \frac{(X^{k-1}Y)^{1/2}}{QR}\sum_{i=1}^3\grT_i,
\end{equation}
where
\begin{align*}
\grT_1&=\calT_0^{-1}\sum_{r\sim R}[g,r]^{1+\nu-s/2}\sideset{}{^*}\sum_{\chi \nmod{r}}
\underset{|\tau|\le \calT_0}\int |F(\kap+i\tau)|\d\tau ,\\
\grT_2&=\sum_{r\sim R}[g,r]^{1+\nu-s/2}\sideset{}{^*}\sum_{\chi \nmod{r}}
\underset{\calT_0<|\tau|\le \calT_3}\int \frac{|F(\kap+i\tau)|}{|\tau|}
\d\tau \\ 
\grT_3&=\sum_{r\sim R}[g,r]^{1+\nu-s/2}\sideset{}{^*}\sum_{\chi \nmod{r}}1.
\end{align*}

\par We estimate the terms $\grT_i$ in turn. Observe first that 
\begin{equation}\label{6.4}
\calT_0^{-1}=\log (\calY/\calX)\ll RQv^{-1}\ll RQX^{-k}.
\end{equation}
Recall once again our assumption that $s\ge 5$. Then an application of Lemma \ref{lemma6.2} yields the 
bound
\begin{align*}
\grT_1&\ll g^{1+\nu-s/2}\calT_0^{-1}\sum_{\substack{1\le l\le 2R\\ l|g}}(R/l)^{1+\nu-s/2}
(l^{-1}R^2\calT_0X^{11/20}+X)L^c\\
&\ll g^{1+\nu-s/2+\eps}L^c\left( RX^{11/20}+RQX^{1-k}\right). 
\end{align*}
On recalling the definition (\ref{2.3}) of $P$ and $Q$, and noting that by hypothesis, one has 
$Y>X^{19/24}>PX^{31/40}$, we therefore deduce that
\begin{align}
\grT_1&\ll g^{1+2\nu-s/2}RQX^{1-k}L^c(1+PX^{31/20}Y^{-2})\notag \\
&\ll g^{1+2\nu-s/2}RQX^{1-k}L^c.\label{6.5}
\end{align}

\par In order to estimate $\grT_2$, we introduce the auxiliary function
$$M(l,R,\Tet,X)=\sum_{\substack{r\sim R\\ l|r}}\,\sideset{}{^*}\sum_{\chi\nmod{r}}
\int_{\Tet}^{2\Tet}|F(\kap+i\tau,\chi)\d\tau .$$
Then by dividing the range of integration into dyadic intervals, one finds in a similar manner to the 
treatment of $\grT_1$ that
\begin{align*}
\grT_2&\ll g^{1+\nu-s/2}L^c\sum_{\substack{1\le l\le 2R\\ l|g}}(R/l)^{1+\nu-s/2}
\max_{\calT_0\le \Tet\le \calT_3}\Tet^{-1}M(l,R,\Tet,X)\\
&\ll g^{1+\nu-s/2}\sum_{\substack{1\le l\le 2R\\ l|g}}(R/l)^{1+\nu-s/2}(l^{-1}R^2X^{11/20}+
\calT_0^{-1}X)L^{2c},
\end{align*}
whence, by reference to (\ref{6.4}), one obtains the bound
\begin{equation}\label{6.6}
\grT_2\ll g^{1+2\nu-s/2}RQX^{1-k}L^c.
\end{equation}

\par Finally, one has
\begin{align}
\grT_3&\ll \sum_{r\sim R}[g,r]^{1+\nu-s/2}r\ll g^{1+\nu-s/2}\sum_{\substack{1\le l\le 2R\\ l|g}}
(R/l)^{1+\nu-s/2}R\notag \\
&\ll g^{1+2\nu-s/2}RQX^{1-k}(PX/Y^2)\ll g^{1+2\nu-s/2}RQX^{1-k}.\label{6.7}
\end{align}

\par By substituting the estimates (\ref{6.5}), (\ref{6.6}) and (\ref{6.7}) into (\ref{6.3}), we 
conclude that the estimate (\ref{6.1}) does indeed hold, and thus the conclusion of 
Lemma \ref{lemma5.2} is finally confirmed.

\section{The estimation of $J_\nu(g)$} The argument of the proof of Lemma \ref{lemma5.2} adapts with 
only modest complications to establish Lemma \ref{lemma5.3}, though making use of the first estimate of 
Lemma \ref{lemma6.1} in place of the second, and we suppress details of the necessary adjustments in 
the interests of concision. We refer the reader to \cite[\S6]{LX2007} and \cite[\S6]{ST2009} for the 
necessary details. One issue deserves additional attention, this being associated with the proof of the 
second estimate $J_\nu(1)\ll YL^{-B}$ of Lemma \ref{lemma5.3}. When $g=1$, we divide the 
summation over $r$ in the definition of $J_\nu(1)$ into dyadic intervals $r\sim R$, and then consider 
seperately the situations with $R\le L^B$ and $L^B<R\le P$, in which $B$ is a positive number depending 
only on $A$. Since $g=1$, in the latter situation we may extract a factor $R^{-1/4}$ from the estimates 
involving summations over $r$ without damaging convergence, and hence replace the bound 
$g^{1+2\nu-s/2}YL^c$ by $g^{1+2\nu-s/2}YL^cR^{-1/4}\ll YL^{-A}$ in the estimates that result. 
Thus it suffices to restrict attention to the situation with $1\le R\le L^B$.\par

When $1\le u\le n$, $r\le 2L^B$ and $\chi$ is a Dirichlet character modulo $r$, we may make use of the 
explicit formula 
\begin{equation}\label{7.1}
\sum_{1\le m\le u}\Lam (m)\chi(m)=\del_\chi u-\sum_{|\gam|\le T}\frac{u^\rho}{\rho}+O\left( 
\left( \frac{u}{T}+1\right) \log^2(ruT)\right) ,
\end{equation}
where we denote by $\rho=\bet+i\gam$ a typical non-trivial zero of the Dirichlet $L$-function $L(s,\chi)$, 
and $T$ is a parameter with $2\le T\le u$. Taking $T=X^{5/12-\nu}$ and substituting (\ref{7.1}) into 
the formula for $\What(\chi,\lam)$, we deduce that whenever $|\lam|\le (rQ)^{-1}$, one has
\begin{align*}
\What(\chi,\lam)&=\int_{N_1}^{N_2}e(u^k\lam)\sum_{|\gam|\le T}u^{\rho-1}\d u+
O\left( (1+X^{k-1}Y|\lam|)XT^{-1}L^2\right) \\
&\ll Y\sum_{|\gam|\le T}X^{\bet-1}+O(X^kYQ^{-1}T^{-1}L^2).
\end{align*}
Recalling (\ref{2.3}) once again, and taking $\nu>0$  and $\del>0$ sufficiently small in terms of 
$\tet$, one finds that
$$X^kQ^{-1}T^{-1}\ll X^2Y^{-2}PT^{-1}\ll X^{\nu+2K\del-2(\tet-19/24)}\ll X^{-2\nu/3},$$
whence
$$\What(\chi,\lam)\ll Y\sum_{|\gam|\le T}X^{\bet -1}+O(X^{-\nu/2}Y).$$

\par Now put $\eta(T)=c(\log T)^{-4/5}$, with $c$ a suitably small positive number. Then by Page's 
Lemma (see \cite[Corollary 11.10]{MV2007}), one sees that $\underset{\chi\nmod{r}}\prod L(s,\chi)$ has 
no zeros in the region $\sig\ge 1-\eta(T)$ and $|t|\le T$, except potentially for a single Siegel zero. But 
as a consequence of Siegel's theorem (see \cite[Corollary 11.15]{MV2007}), no Siegel zero plays any role 
when $r\sim R\le L^B$. Thus, on making use of the zero-density estimates for Dirichlet $L$-functions 
of large sieve type given in \cite[equation (1.1)]{Hux1974}, we deduce that
\begin{align*}
\sum_{r\sim R}\ \sideset{}{^*}\sum_{\chi\nmod{r}}\sum_{|\gam|\le T}X^{\bet-1}&\ll 
L^c\int_0^{1-\eta(T)}T^{12(1-\alp)/5}X^{\alp-1}\d\alp ,\\
&\ll L^c\int_0^{1-\eta(T)}X^{12(\alp-1)\nu/5}\d\alp ,
\end{align*}
whence
$$\sum_{r\sim R}\ \sideset{}{^*}\sum_{\chi\nmod{r}}\sum_{|\gam|\le T}X^{\bet-1}\ll L^c 
X^{-12\eta(T)\nu/5}\ll \exp(-c' L^{1/5}),$$
for a suitable positive number $c'$. Assembling these estimates together, we conclude that
$$\sum_{r\le P}\ \sideset{}{^*}\sum_{\chi\nmod{r}}\max_{|\lam|\le 1/(rQ)}|\What (\chi,\lam)|\ll 
YL^{-A},$$
where $A>0$ is arbitrary. This completes the proof of the second estimate of Lemma \ref{lemma5.3}.

\section{Exceptional sets} The basic exceptional set conclusions embodied in Theorem \ref{theorem1.2} 
follow by the standard arguments employing Bessel's inequality. Let $\calZ$ denote the set of integers $n$ 
with $N\le n\le N+X^{k-1}Y$ and $n\equiv s\mmod{R}$ (and, in case $k=3$ and $s=7$, satisfying also 
$9\nmid n$), for which the equation
\begin{equation}\label{8.1}
n=p_1^k+p_2^k+\ldots +p_s^k
\end{equation}
has no solution in prime numbers $p_j$ with $|p_j-X|<Y$ $(1\le j\le s)$. In addition, put $Z=\text{card}
(\calZ)$. Then we find from Bessel's inequality and (\ref{2.4}) that
$$\sum_{n\in \calZ}|\rho_s(n;\grm)|^2\le \sum_{N\le n\le N+X^{k-1}Y}|\rho_s(n;\grm)|^2
\le \int_\grm |f(\alp)|^{2s}\d\alp .$$
We therefore conclude from Propositions \ref{prop2.2} and \ref{prop2.3}, together with (\ref{2.9}), that 
whenever $s>t_k$, then
\begin{align}
\sum_{n\in \calZ}|\rho_s(n;\grm)|^2&\le \Bigl( \sup_{\alp\in \grm}|f(\alp)|\Bigr)^{2s-2t_k}
\int_\grm |f(\alp)|^{2t_k}\d\alp \notag \\
&\ll (X^{\eps-2\del}Y)^{2s-2t_k}Y^{2t_k-1}X^{1-k+\del}\ll Y^{2s-1}X^{1-k-\del}.\label{8.2}
\end{align}
Meanwhile, it follows from Proposition \ref{prop2.1} and the lower bound in (\ref{2.6}) that whenever 
$s\ge \min\{5,k+2\}$, then
\begin{equation}\label{8.3}
\sum_{n\in \calZ}|\rho_s(n;\grM)|^2\gg Z(Y^{s-1}X^{1-k})^2.
\end{equation}
Since for $n\in \calZ$, one necessarily has
$$\rho_s(n;\grM)+\rho_s(n;\grm)=\rho_s(n)=0,$$
it follows from (\ref{8.2}) and (\ref{8.3}) that
$$Z(Y^{s-1}X^{1-k})^2\ll Y^{2s-1}X^{1-k-\del},$$
whence $Z\ll X^{k-1-\del}Y$.\par

Since the number of integers $n$ satisfying $N\le n\le N+X^{k-1}Y$, together with other associated 
conditions described in the opening paragraph, is of order $X^{k-1}Y$, it follows that for almost all of 
these integers one has $\rho_s(n)\ne 0$. We observe also that when $|p_j-X|<Y$ and 
$N\le n\le N+X^{k-1}Y$, one has also $|p_j-(n/s)^{1/k}|\ll Y$, and thus the conclusion of Theorem 
\ref{theorem1.2} is confirmed. This completes our proof of Theorem \ref{theorem1.2}.\par

It may be worth observing at this point that, since one has $|p_j-X|<Y$ in the representation (\ref{8.1}), 
then necessarily one has $|n-sX^k|\ll X^{k-1}Y$, or equivalently $|n-N|\ll X^{k-1}Y$. We therefore see 
that, in order for an exceptional set estimate to constitute a non-trivial assertion to the effect that almost all 
eligibible integers are represented in the form (\ref{8.1}), one must establish an estimate of the shape 
$Z=o(X^{k-1}Y)$. Assertions in \cite{LW2008} and \cite {Zha2010} are only slightly stronger than 
$Z=o(X^k)$, and consequently yield non-trivial exceptional sets only when $Y$ is very nearly as large as 
$X$. This explains the origin of the infelicities noted in the discussion following the statement of Theorem 
\ref{theorem1.2} above.\par

We finish this section by noting that the ideas underlying the proof of Proposition \ref{prop2.2} may be 
used to good effect in sharpening estimates for exceptional sets. In order to illustrate such ideas, we begin 
by recording a lemma of use in estimating exceptional sets for sums of squares of almost equal primes. 
Here, we adopt the notation of \S\S2 and 3.

\begin{lemma}\label{lemma8.1} Let $\calZ\subseteq [N,N+X^{k-1}Y]\cap \dbZ$, and put
$$K(\alp)=\sum_{n\in \calZ}e(n\alp).$$
Then one has
$$\int_0^1|f(\alp)^4K(\alp)^2|\d\alp \ll X^\eps (Y^2Z+Y^3X^{-1}Z^2).$$
\end{lemma}

\begin{proof} By adjusting the value of $Y$ by an amount at most $O(1)$, we may suppose that $X$ is 
an integer, and consequently we are free to suppose the latter in the proof of this lemma. Next, by 
orthogonality, the mean value in question is bounded above by the number of integral solutions of the 
equation
\begin{equation}\label{8.4}
(X+y_1)^2+(X+y_2)^2-(X+y_3)^2-(X+y_4)^2=n_1-n_2,
\end{equation}
with $|y_i|\le Y$ $(1\le i\le 4)$ and $n_1,n_2\in \calZ$, and with each solution counted with weight
$$\prod_{i=1}^4 \log (X+y_i).$$
Let $T_1$ denote the number of such solutions in which $n_1=n_2$, and let $T_2$ denote the 
corresponding number of solutions with $n_1\ne n_2$. Then we have
\begin{equation}\label{8.5}
\int_0^1|f(\alp)^4K(\alp)^2|\d\alp \ll X^\eps (T_1+T_2).
\end{equation}

\par By orthogonality and an application of Hua's lemma, we find from (\ref{8.4}) that
\begin{equation}\label{8.6}
T_1\le \text{card}(\calZ)\int_0^1|f(\alp)|^4\d\alp \ll ZY^{2+\eps}.
\end{equation}
Consider then a solution $\bfy,\bfn$ counted by $T_2$. Write $m=m(n_1,n_2)$ for the integer closest to 
$(n_1-n_2)/(2X)$. Then since it follows from (\ref{8.4}) that
$$2X(y_1+y_2-y_3-y_4)+(y_1^2+y_2^2-y_3^2-y_4^2)=n_1-n_2,$$
we find that
$$|y_1+y_2-y_3-y_4-m|\le 1+Y^2/X\le 2Y^2/X.$$
Thus we discern that for each fixed choice of $n_1$ and $n_2$ with $n_1,n_2\in \calZ$ and 
$n_1\ne n_2$, there exists an integer $h$ with $|h-m|\le 2Y^2/X$ and
\begin{equation}\label{8.7}
\left. \begin{aligned}
y_1^2+y_2^2-y_3^2-y_4^2&=n_1-n_2-2Xh\\
y_1+y_2-y_3-y_4&=h
\end{aligned} \right\} .
\end{equation}
We divide the solutions $\bfy,\bfn$ counted by $T_2$ into two types. Let $T_3$ denote the number of 
the solutions counted by $T_2$ in which
\begin{equation}\label{8.8}
n_1-n_2=h(2X+2y_3+h),
\end{equation}
and let $T_4$ denote the corresponding number of solutions in which the latter equation does not hold.
\par

Given a fixed choice of $n_1,n_2\in \calZ$ with $n_1\ne n_2$, a divisor function estimate shows that the 
number of possible choices for $h$ and $y_3$ satisfying (\ref{8.8}) is $O(X^\eps)$. Fix any one such 
choice, and eliminate $y_4$ between the equations (\ref{8.7}). We deduce that
$$(y_1+y_2-y_3-h)^2=(y_1^2+y_2^2-y_3^2)-(n_1-n_2-2Xh),$$
whence
\begin{equation}\label{8.9}
2(y_1-y_3-h)(y_2-y_3-h)=h(2X+2y_3+h)-(n_1-n_2).
\end{equation}
Thus, for solutions counted by $T_3$ in which the right hand side vanishes, one has $y_1=y_3+h$ or 
$y_2=y_3+h$. In the former case, the value of $y_1$ is fixed, and it follows from the linear equation of 
(\ref{8.7}) that $y_2=y_4$. A symmetrical argument yields a symmetrical conclusion in the latter case, 
and thus we deduce that
\begin{equation}\label{8.10}
T_3\ll X^\eps YZ^2\ll X^\eps (Y^3/X)Z^2.
\end{equation}
Meanwhile, for solutions counted by $T_4$ in which the right hand side of (\ref{8.9}) is non-zero, we find 
by a divisor estimate that for each of the $O(Y^2/X)$ possible choices for $h$, and each of the $O(Y)$ 
possible choices for $y_3$, there are $O(X^\eps)$ possible choices for $y_1-y_3-h$ and $y_2-y_3-h$. 
Given any fixed such choices, we find that $y_1$ and $y_2$ are fixed, and then $y_4$ is also fixed by 
virtue of the linear equation in (\ref{8.7}). Thus we conclude that
\begin{equation}\label{8.11}
T_4\ll X^\eps Y(Y^2/X)Z^2.
\end{equation}
By combining (\ref{8.6}), (\ref{8.10}) and (\ref{8.11}), and substituting into (\ref{8.5}), the conclusion of 
the lemma now follows.
\end{proof}

Equipped with this lemma, we may establish a powerful estimate for the exceptional set associated with six 
almost equal squares of prime numbers. When $Y\ge 1$, denote by $E_6(N;Y)$ the number of positive 
integers $n$, with $n\equiv 6\mmod{24}$ and $|n-N|\le XY$, for which the equation 
$n=p_1^2+p_2^2+\ldots +p_6^2$ has no solution in prime numbers $p_j$ with 
$|p_j-(n/6)^{1/2}|<Y$ $(1\le j\le 6)$.

\begin{theorem}\label{theorem8.2}
Suppose that $Y\ge X^{19/24+\eps}$, for some positive number $\eps$. Then there is a positive number 
$\del$ for which $E_6(N;Y)\ll Y^{-1}X^{1-\del}$.
\end{theorem}

\begin{proof} Let $\calZ$ denote the set of integers counted by $E_6(N;Y)$. Then for each $n\in \calZ$ 
one has
$$\int_\grm f(\alp)^6e(-n\alp)\d\alp =-\int_\grM f(\alp)^6e(-n\alp)\d\alp ,$$
whence Proposition \ref{prop2.1} and the associated discussion yields the lower bound
$$\biggl| \int_\grm f(\alp)^6K(-\alp)\d\alp \biggr| \gg \sum_{n\in \calZ} Y^5X^{-1}=ZY^5X^{-1}.$$
By Schwarz's inequality, we therefore deduce that
$$ZY^5X^{-1}\ll \biggl( \int_0^1 |f(\alp)^4K(\alp)^2|\d\alp \biggr)^{1/2} 
\biggl( \int_\grm |f(\alp)|^8\d\alp \biggr)^{1/2}.$$
An application of Proposition \ref{prop2.2} together with (\ref{2.9}) shows that
\begin{align*}
\int_\grm |f(\alp)|^8\d\alp &\le \Bigl( \sup_{\alp\in \grm}|f(\alp)|\Bigr)^2\int_0^1|f(\alp)|^6\d\alp \\
&\ll (YX^{-2\del})^2Y^5X^{\eps-1},
\end{align*}
and so we deduce from Lemma \ref{lemma8.1} that
\begin{align*}
ZY^5X^{-1}&\ll X^\eps (Y^2Z+Y^3X^{-1}Z^2)^{1/2}(Y^7X^{-1-2\del})^{1/2}\\
&\ll Y^{9/2}X^{-(1+\del)/2}Z^{1/2}+Y^5X^{-1-\del/2}Z.
\end{align*}
Consequently, one sees that $Z\ll X^{1-\del}Y^{-1}$, and the proof of the theorem is complete.
\end{proof}

\bibliographystyle{amsbracket}
\providecommand{\bysame}{\leavevmode\hbox to3em{\hrulefill}\thinspace}

\end{document}